\documentclass[pdftex,mathpazo]{cicp}
\usepackage{graphicx}
\usepackage{color}
\usepackage{amsthm}
\usepackage{amsmath}
\usepackage{amsfonts}
\usepackage{fancyhdr}
\usepackage{hyperref}

\newcommand{\etal}{{\em et.al.}}
\newcommand{\fpi}{\frac{1}{2\pi}}
\newcommand{\ipi}{\int_0^{2\pi}}
\newcommand{\avgint}{\fpi \ipi}
\newcommand{\tint}{\int_{t}^{t + \tau}}

\newcommand{\A}{{\mathbf A}}
\newcommand{\flux}{\boldsymbol{\mathcal{F}}}
\newcommand{\source}{\boldsymbol{\mathcal{S}}}
\newcommand{\flus}{\flux^*}
\newcommand{\sours}{\source^*}
\newcommand{\hflux}{\boldsymbol{\mathcal{H}}}
\newcommand{\hfluxpd}{\boldsymbol{\mathcal{H}}^+_{E,\text{drain}}}
\newcommand{\vx}{{\vec x}}
\newcommand{\vv}{{\vec v}}
\newcommand{\vu}{{\mathbf u}}
\newcommand{\vU}{{\mathbf U}}
\newcommand{\vw}{{\mathbf w}}
\newcommand{\vW}{{\mathbf W}}
\newcommand{\vE}{{\mathbf E}}
\newcommand{\vt}{{\mathbf t}}

\newcommand{\tQ}{{\tilde Q}}
\newcommand{\htot}{H}
\newcommand{\dt}{\Delta t}
\newcommand{\dtd}{\dt_{C_i,\text{drain}}}
\newcommand{\dx}{\Delta x}
\newcommand{\eps}{\varepsilon}
\newcommand{\luka}{Luk\'a\v{c}ov\'a}
\newcommand{\ec}{\vE^{\text{const}}}
\newcommand{\eb}{\vE^{\text{bilin}}}

\DeclareMathOperator{\sgn}{sgn}
\DeclareMathOperator{\sech}{sech}

\theoremstyle{definition}
\newtheorem{alg}{Algorithm}[section]

\graphicspath{{./figures/}}
\DeclareGraphicsExtensions{.pdf, .pdftex, .eps, .jpg, .png, {}}
\begin{document}
  
\title[FVEG schemes for the SWE with dry beds]{Finite Volume Evolution
  Galerkin Methods for the Shallow Water Equations with Dry Beds}

\author[A. Bollermann, S. Noelle and M. \luka-Medvid'ov\'{a}]{Andreas
  Bollermann\affil{1}\comma\corrauth,
Sebastian Noelle\affil{1} and Maria \luka-Medvid'ov\'{a}\affil{2}}

\email{{\tt bollermann@igpm.rwth-aachen.de} (A.~Bollermann)}

\address{\affilnum{1}\ IGPM, RWTH Aachen, Templergraben 55, 52062 Aachen,
  Germany. \\ \affilnum{2}\ Department of Mathematics, University of Technology
  Hamburg, Schwarzenbergstra{\ss}e 95, 21073 Hamburg, Germany}

\begin{abstract}
  We present a new Finite Volume Evolution Galerkin (FVEG) scheme for
  the solution of the shallow water equations (SWE) with the bottom
  topography as a source term. Our new scheme will be based on the
  FVEG methods presented in (\luka, Noelle and Kraft, {\em
    J. Comp. Phys.} 221, 2007), but adds the possibility to handle dry
  boundaries. The most important aspect is to preserve the positivity
  of the water height. We present a general approach to ensure this
  for arbitrary finite volume schemes. The main idea is to limit the
  outgoing fluxes of a cell whenever they would create negative water
  height. Physically, this corresponds to the absence of fluxes in the
  presence of vacuum. Well-balancing is then re-established by
  splitting gravitational and gravity driven parts of the
  flux. Moreover, a new entropy fix is introduced that improves the
  reproduction of sonic rarefaction waves.
\end{abstract}

\pac{02.60.Cb,  47.11.Df, 92.10.Sx}

\maketitle

\section{Introduction}
\label{sec:intro}

The shallow water equations (SWE) are a mathematical model for the
movement of water under the action of gravity. Mathematically spoken,
they form a set of hyperbolic conservation laws, which can be extended
by source terms like the influence of the bottom topography, friction
or wind forces. In this case, we will speak of a balance law. For
simplicity, this work will consider the variation of the bottom as the
only source term.

Many important properties of the model rely on the fact that the water
height is strictly positive. Despite this, typical relevant problems
include the occurrence of dry areas, like dam break problems or the
run-up of waves at a coast, with tsunamis as the most impressive
example. So for simulations of these problems, we have to develop
numerical schemes that can handle the (possibly moving) shoreline in a
stable and efficient way.  Another crucial point in solving balance
laws is the treatment of the source terms. For precise solutions, it
is necessary to evaluate the source term in such a way that certain
steady states are kept numerically, i.e. the numerical flux and
the numerical source term cancel each other exactly for equilibrium
solutions.


In the last years, many groups contributed to the solution of the
difficulties described above. In
\cite{AudusseBouchutBristeauEtAl2004}, Audusse \etal{} proposed a
reconstruction procedure where the free surface and water height are
reconstructed and the bottom slopes are computed from these. This
guarantees the positivity of the water height and gives a
well-balanced scheme at the same time. Begnudelli and Sanders
developed a scheme for triangular meshes including scalar transports
in \cite{BegnudelliSanders2006}. They proposed a strategy how to
exactly represent the free surface in partially wetted cells, leading
to improved results at the wetting/drying front. In
\cite{BrufauVazquez-CendonGarcia-Navarro2002}, Brufau \etal{} analyse
how to deal with flow on an adverse slope. They locally modify the
bottom topography in certain situations to avoid unphysical run-ups or
wave creation at the dry boundary. Gallardo \etal{} discussed various
solutions of the Riemann problem at the front and used them in a
modified Roe scheme. They then used the local hyperbolic harmonic
method from Marquina (cf. \cite{Marquina1994}) in the reconstruction
step to achieve higher order, see
\cite{GallardoParesCastro2007}. Kurganov and Petrova proposed a
central-upwind scheme that is well-balanced and positivity preserving
in \cite{KurganovPetrova2007}. It is based on a continuous, piecewise
linear approximation of the bottom topography and performs the
computation in terms of the free surface instead of the relative water
height to simplify the well-balancing. The last feature is also a
building block in the work of Liang and Marche
\cite{LiangMarche2009}. They also provide a method to extend this
well-balancing feature to situations including wetting/drying
fronts. Liang and Borthwick \cite{LiangBorthwick2009} used adaptive
quad-tree grids to improve the efficiency of their schemes. Wetting and
drying effects are handled as well as friction terms. In the context
of residual distribution methods, Ricchiuto and Bollermann developed a
positivity preserving and well-balanced scheme for unstructured
triangulations \cite{RicchiutoBollermann2009}.

The finite volume evolution Galerkin (FVEG) methods developed by
\luka, Morton and Warnecke, cf. \cite{LukacovaMortonWarnecke2004,
  LukacovaSaibertovaWarnecke2002, LukacovaNoelleKraft2007}, have been
successfully applied to the SWE in
\cite{LukacovaNoelleKraft2007}. They are based on the evaluation of so
called {\em evolution operators} which predict values for the finite
volume update. Thanks to these operators, the schemes take into
account all directions of wave propagation, enabling them to precisely
catch multidimensional effects even on Cartesian grids.  These schemes
show a very good accuracy even on relatively coarse meshes compared to
other state of the art schemes and they are also competitive in terms
of efficiency (cf. \cite{LukacovaNoelleKraft2007}).

However, the existing FVEG schemes are not able to deal with dry
boundaries. Thus in this work we will present a method to preserve the
positivity of the water height with an arbitrary finite volume
method. To achieve this, we reduce the outflow on draining cells such
that the water height does not become negative. We will then provide
the means to preserve the {\em well-balancing} property under the
presence of dry areas, and apply both techniques to a new FVEG method.
In addition, we present a new entropy fix for the FVEG schemes that
improves the reproduction of sonic rarefaction waves.

We start our paper with a short presentation of the SWE in
Section~\ref{sec:SWE}. Section~\ref{sec:FVEG} describes the FVEG
method we will start from. The arising difficulties by introducing dry
areas and means to overcome them are described in
Section~\ref{sec:dry}, which is the main part of the paper. Finally,
in Section~\ref{sec:res}, we will show selected numerical test cases
that demonstrate the performance of our schemes.

\section{The Shallow Water Equations}
\label{sec:SWE}
\subsection{Balance Law Form}

We consider the shallow water system in balance form
\begin{equation}
\label{eq:cSWE}
\frac{\partial \vu}{\partial t} + \nabla\cdot \flux (\vu) =
-\source(\vu, \vec x).
\end{equation}
The conserved variables and the flux are given by
\begin{equation}
\label{eq:watervariables}
\vu = \begin{pmatrix} h \\ hv_1 \\ hv_2 \end{pmatrix},
 \quad 
\flux(\vu) = \left(\flux_{1}(\vu)\, \flux_{2}(\vu) \right) = 
\begin{pmatrix}
hv_1 & hv_2 \\
hv_1^2 + g\frac{h^2}{2} & hv_1v_2 \\
hv_1v_2 & hv_2^2 + g\frac{h^2}{2}
\end{pmatrix},
\end{equation}
where $h$ denotes the relative water height, $\vv = (v_1,v_2)^T$ the flow
speed and $g$ the (constant) gravity acceleration. The source term
$\source(\vu, \vec x)$ is given by
\begin{equation}
  \label{eq:source}
\source (\vu, \vec x) = gh
\begin{pmatrix}
0 \\ \frac{\partial
      b(\vec x)}{\partial x_1} \\ \frac{\partial b(\vec x)}{\partial x_2}  
\end{pmatrix}
\end{equation}
with $b(\vec x)$ the local bottom height.  We also introduce the {\em
  free surface level}, or total water height, 
\begin{equation}
\label{eq:htot}
\htot(\vec x)  = h(\vec x) + b(\vec x)
\end{equation}
and the so-called {\em speed of sound}
\begin{equation}
  \label{eq:sos}
  c = \sqrt{gh}.
\end{equation}
This is the velocity of the gravity waves and should not be confused
with the physical sound speed in air.

\subsection{Quasi-linear Form}
\label{sec:quasi-linear}

For the derivation of the evolution operators in Section~\ref{sec:op},
it is helpful to rewrite (\ref{eq:cSWE}) in primitive variables. The
system then takes the form
\begin{equation}
  \label{eq:pSWE}
\vw_t + \A_1(\vw)\vw_{x_1} + \A_2(\vw)\vw_{x_2} = \vt
\end{equation}
with
\begin{equation}
\label{eq:pvar}
\vw = \begin{pmatrix}
    h \\ v_1 \\ v_2
  \end{pmatrix}, \quad
\A_1 = \begin{pmatrix}
    v_1 & h & 0 \\ g & v_1 & 0 \\ 0 & 0 & v_1
  \end{pmatrix}, \quad
\A_2 = \begin{pmatrix}
    v_2 & 0 & h \\ 0 & v_2 & 0 \\ g & 0 & v_2
  \end{pmatrix}
\end{equation}
and the source term
\begin{equation}
  \vt = \begin{pmatrix}
    0 \\ -gb_{x_1} \\ -gb_{x_2}
  \end{pmatrix}.
\end{equation}

For each angle $\theta \in[0,2\pi)$ we define  the direction $ \vec
\xi(\theta) := (\cos\theta, \sin \theta)$.  As system (\ref{eq:cSWE})
is hyperbolic, for each of these directions and a fixed $\vw$
the matrix
\begin{equation}
  \label{eq:jac}
\A(\vw) = \vec\xi_1  \A_1(\vw) + \vec\xi_2  \A_2(\vw)  
\end{equation}
has real eigenvalues
\begin{equation}
  \label{eq:ev}
  \lambda_1 = \vv \cdot \vec \xi - c, \quad
  \lambda_2 = \vv \cdot \vec \xi, \quad
  \lambda_3 = \vv \cdot \vec \xi + c
\end{equation}
 and a full set of linearly independent eigenvectors
\begin{equation}
  \label{eq:evec}
r_1 = \begin{pmatrix}
    -1 \\ g \frac{\cos \theta}{c} \\ g \frac{\sin \theta}{c}
  \end{pmatrix}, \quad
r_2 = \begin{pmatrix}
    0 \\  \sin \theta \\ -\cos \theta
  \end{pmatrix}, \quad
r_3 = \begin{pmatrix}
    1 \\ g \frac{\cos \theta}{c} \\ g \frac{\sin \theta}{c}
  \end{pmatrix}.
\end{equation}

\subsection{Lake at Rest}
\label{sec:lar}

A trivial, but nevertheless important solution to (\ref{eq:cSWE}) is
the lake at rest situation, where the water is steady and the free
surface level is constant, i.e.  we have
\begin{equation}
  \label{eq:lar}
   \vv = (0,0)^T \text{ and } \htot(\vec x) = H_0.
\end{equation}
From (\ref{eq:htot}) we immediately get
\begin{equation}
  \label{eq:bal}
  \nabla h = - \nabla b
\end{equation}
and therefore (with (\ref{eq:cSWE})~--~(\ref{eq:source}) and $\vv =
(0,0)^T$)
\begin{equation}
  \label{eq:wb}
  \begin{pmatrix}
      0 \\
      g\frac{h^2}{2} \\
      0
    \end{pmatrix}_{x_1} +
  \begin{pmatrix}
      0\\
      0\\
      g\frac{h^2}{2}
    \end{pmatrix}_{x_2} =
  -gh \begin{pmatrix} 0 \\ 
      \frac{\partial  b(\vec x)}{\partial x_1} \\ 
      \frac{\partial b(\vec x)}{\partial x_2} 
    \end{pmatrix}.
\end{equation}
A scheme fulfilling a discrete analogon of (\ref{eq:wb}) exactly is
called {\em well-balanced}.

\section{FVEG Schemes}
\label{sec:FVEG}

Finite volume schemes are very popular for solving hyperbolic
conservation laws for several reasons. They represent the underlying
physics in a natural way and can be implemented very
efficiently. Nevertheless, nearly all of them are based on the
solution of one-dimensional Riemann problems and therewith a
dimensional splitting. This introduces some sort of a bias: Wave
propagation aligned with the grid is very well represented, whereas
waves oblique to the grid cannot be caught as accurate.

In the last decade \luka{} \etal{} developed a class of finite
volume evolution Galerkin schemes, see e.g.
\cite{LukacovaMortonWarnecke2000, LukacovaSaibertovaWarnecke2002,
  LukacovaWarneckeZahaykah2004}. The FVEG scheme is a
predictor-corrector method: In the predictor step a multidimensional
evolution is done, the corrector step is a finite volume update.

In this section we will recall the second order scheme presented in
\cite{LukacovaNoelleKraft2007}. This method will be the starting point
for our extensions for computations including dry beds in
Section~\ref{sec:dry}. Therefore we concentrate on the properties
playing a role in this context and limit ourselves to the main ideas
otherwise.

\subsection{Finite Volume Update}
\label{sec:fv}

For our computations, we use Cartesian grids, i.e. we divide our
computational domain $\Omega$ in rectangular cells $C_i$, separated by
edges $E$. On the edges, we have quadrature points $\vx_k$. The
subscript $i$ will always refer to a cell, whereas $k$ as a subscript
is used as a global index for quadrature points. If we talk about the
local quadrature points on a single edge, we use the index $j$
instead.

On each cell we define the initial value at as
\begin{equation}
  \label{eq:ca}
  \vu_i^0 := \vu_i(0) \approx \frac{1}{|C_i|}\int_{C_i} \vu(\vec x, 0)\,d\vec x
\end{equation}
where we use a Gaussian quadrature to approximate the
integral. Integrating (\ref{eq:cSWE}) on each cell, we can then define
the update as
\begin{equation}
  \label{eq:update}
  \vu_i^{n+1} = \vu_i^n  - \frac{1}{|C_i|}
  \int_{t^n}^{t^{n+1}} \left( \int_{\partial  C_i} 
    \flux(\vu(\vec x, t))\cdot \vec n \, d\vec x
    + \int_{C_i} \source(\vu(\vx, t),\vx)\, d\vec x\right)dt
\end{equation}
using the Gauss theorem. Here $\vu_i^n$ denotes cell average in $C_i$
at time $t^n$ and $\vec n$ is the outer normal. The solution on the
whole domain at time $t^n$ is then defined as
\begin{equation}
  \label{eq:globalu}
  \vU^n(\vx) := \vU(\vx, t^n) = \vu_i^n, \quad \vx \in C_i.
\end{equation}

For an approximation of (\ref{eq:update}), on each edge we define
three quadrature points $\vec x_{j}, j=1,2,3$, see Fig.~\ref{fig:quad}. 
\begin{figure}
  \centering
  \includegraphics{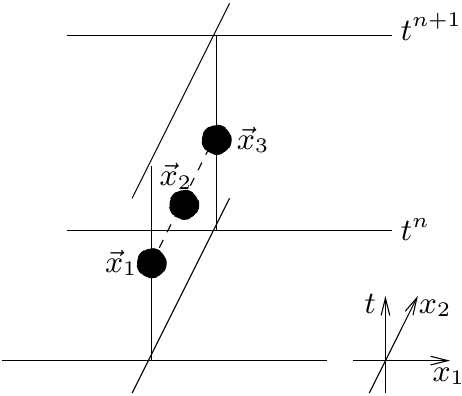}
  \caption{Quadrature points $\vec x_j$ for finite volume update}
  \label{fig:quad}
\end{figure}
These quadrature points are located on the vertices ($j=1,3$) and the
centre ($j=2$) of an edge. The flux over the edge is approximated by
using midpoint rule in time and Simpson's rule in space, hence we will
use the evolution operators from Section~\ref{sec:op} to predict point
values at the quadrature points at time $t^{n+1/2}$. The flux over an
edge is then defined as
\begin{equation}
  \label{eq:edgeflux}
  \flux_E := \sum_{j=1}^3 \alpha_j \flux(\vu_j^{n+1/2}) \cdot \vec n \approx
  \frac{1}{\dt |E|} \int_{t^n}^{t^{n+1}}\int_E \flux(\vu(\vec x,
  t))\cdot \vec n \,d\vec xdt.
\end{equation}
$\dt = t^{n+1}-t^n$ is the time step, $\vu_j^{n+1/2}$ is an
approximation to $\vu(\vx_j,t^{n+1/2})$ and the $\alpha_j$ represent the
weights of Simpson's rule, i.e. we have $\alpha_{1,3} = \frac{1}{6}$
and $\alpha_2 = \frac{2}{3}$. Finally the source term is discretised
as
\begin{equation}
  \label{eq:dsource}
  \source_i := g  \sum_{j=1}^3 \alpha_j
  \begin{pmatrix}
      0 \\ \frac12(\hat h_j^r + \hat h_j^l)(b_j^r-b_j^l) \\
      \frac12(\hat h_j^t + \hat h_j^b)(b_j^t-b_j^b) 
    \end{pmatrix}
  \approx \frac{1}{\dx \dt } \int_{t^n}^{t^{n+1}} \int_{C_i}
  \source(\vu)\, d\vec x.
\end{equation}
Here, $\dx$ is the length of the element, $\hat h_j$ represents the
first component of $\vu_j^{n+1/2}$ and $b_j = b(\vec x_j)$. The
superscripts stand for the edges surrounding the cell, namely the {\bf
  r}ight, {\bf l}eft, {\bf t}op and {\bf b}ottom
edge. Eqs.~(\ref{eq:update})--(\ref{eq:dsource}) lead to the fully
discrete scheme
\begin{equation}
  \label{eq:fullydiscrete}
  \vu_i^{n+1} = \vu_i^n -  \frac{\dt}{\Delta x} \left[\left( \sum_{E, E
    \subset \partial C_i} \flux_E\right) + \source_i \right].
\end{equation}
The time step is chosen as
\begin{equation}
  \label{eq:ts}
  \dt = \mu \min_{i}\dfrac{\dx}{\max_k |\lambda_k|}
\end{equation}
with $\lambda_k$ the eigenvalues from (\ref{eq:ev}) and $\mu < 1$ a CFL
number. For all the numerical experiments in Section~\ref{sec:res}, we
set $\mu =0.5$.

\subsection{Evolution Operators}
\label{sec:op}

As mentioned before, we use so-called evolution operators to predict
point values of the solution for the quadrature points in
(\ref{eq:edgeflux}). Indeed, a solution of (\ref{eq:pSWE}) can be seen
as a superposition of waves. So for a fixed point $P=(\vx, t)$, we
want to identify all the waves that contribute to the solution
there. This section will describe shortly the evolution operator,
present an exact formulation and give an example of a suitable
approximation allowing an efficient implementation.

The derivation of evolution operators is based on the quasi-linear
form of the system (\ref{eq:pSWE}). For any given point $P$, we
identify a suitable average value $\bar \vw$ and linearise the system
around $P$:
\begin{equation}
  \label{eq:lSWE}
\vw_t + \A_1(\bar \vw)\vw_{x_1} + \A_2(\bar \vw)\vw_{x_2} = \vt.
\end{equation}
The waves we are looking for propagate along the characteristics of
this system. Thus for a fixed direction $\vec \xi(\theta)$, we apply
an one-dimensional characteristic decomposition of the linearised
system. This allows us to identify different wave propagations
corresponding to the eigenvalues (\ref{eq:ev}), the {\em
  bicharacteristics}. The left side of Fig. \ref{fig:cone} shows an
illustration.
\begin{figure}
  \centering
  \hfill
  \includegraphics{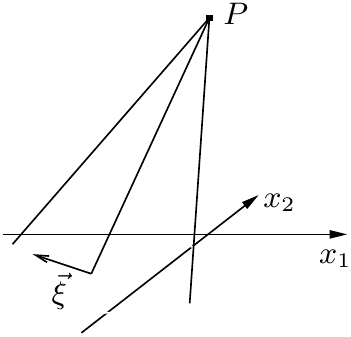}
  \hfill
  \includegraphics{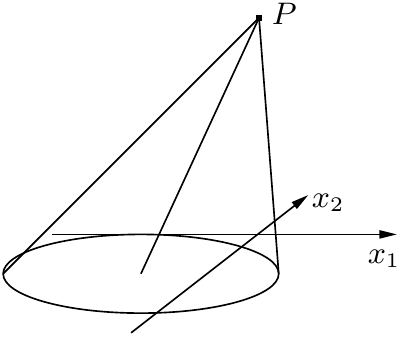}
  \hspace*{\fill}
  \caption{Bicharacteristic decomposition. Left: Bicharacteristic
    curves for a fixed direction $\vec\xi$. Right: Bicharacteristic
    cone.}
  \label{fig:cone}
\end{figure}
Integrating the decomposed system along the bicharacteristics, we get
an integral representation of the solution at point $P$. At this
point, the solution still depends on a particular direction $\vec
\xi(\theta)$ and therefore does not respect waves coming from other
directions. Thus we perform the decomposition for all angles $\theta
\in [0, 2\pi)$ and average the solution at $P$ over $\theta$. This
yields the exact evolution operator of (\ref{eq:lSWE}). The
combination of all bicharacteristics yields the {\em bicharacteristic
  cone} shown in the right picture of Fig.~\ref{fig:cone}. We
introduce the following notation for the peak $P = (\vx, t^n + \tau)$
and points on the sonic cone:
\begin{align}
  Q_0 &:= (x_1 +\tau \bar v_1,\; x_2 + \tau \bar v_2,\; t^n)\label{eq:Q0}\\
  \tQ_0 &:= (x_1 + (t^n + \tau - \tilde t) \bar v_1  \cos\theta,\; x_2 +
  (t^n + \tau - \tilde t) \bar v_2 \sin \theta,\; \tilde t)\label{eq:tQ0}\\
  Q &:= (x_1 + \tau (\bar c + \bar v_1) \cos\theta,\; x_2 
  + \tau (\bar c + \bar v_2) \sin \theta,\; t^n)\label{eq:Q}\\
  \tQ &:= (x_1 + (t^n + \tau - \tilde t)(\bar c + \bar v_1) \cos\theta,\; x_2 +
  (t^n + \tau - \tilde t)(\bar c + \bar v_2) \sin \theta,\; \tilde t)\label{eq:tQ}
\end{align}
$Q_o$ is the centre of the sonic circle at time $t=t^n$, $\tQ_0$
denotes a point on the inner bicharacteristic connecting $P$ and
$Q_0$, $Q$ is a point on the perimeter of the sonic circle at time
$t=t^n$ and $\tQ$ denotes a point on the mantle of the sonic cone at
an arbitrary time $\tilde t\in[t^n,t^n+\tau]$.

After some tedious calculations, see
e.g. \cite{LukacovaNoelleKraft2007}, we get the evolution operators
for the SWE:
\begin{align}
\label{eq:exacth}
  \begin{split}
  h(P) &= \avgint h(Q) - \frac{\bar c}{g} (v_1(Q)
  \cos \theta + v_2(Q) \sin \theta)\; d\theta \\
  &\quad + \frac{\bar c}{2\pi} \tint \ipi (b_{x_1}(\tQ)\cos \theta
  + b_{x_2}(\tQ) \sin\theta)\; d\theta d\tilde t \\
  &\quad -\fpi \tint \frac{1}{t + \tau - \tilde t} \ipi \frac{\bar c}{g}
  (v_1(\tQ) \cos\theta + v_2(\tQ) \sin\theta) \; d\theta d\tilde t
  \end{split}\displaybreak[0] \\
  \label{eq:exactu}
  \begin{split}
    v_1(P) &= \frac{1}{2} v_1(Q_0) + \avgint -\frac{g}{\bar c} h(Q) \cos
    \theta + v_1(Q)\cos^2\theta + v_2(Q) \sin \theta \cos \theta \;
    d\theta \\
    &\quad - \frac{g}{2} \tint h_{x_1}(\tQ_0)+b_{x_1}(\tQ_0) \;
    d\tilde t\\
    &\quad - \frac{g}{2\pi}\tint \ipi (b_{x_1}(\tQ)\cos^2 \theta
    + b_{x_2}(\tQ) \sin\theta\cos\theta)\; d\theta d\tilde t   \\
    &\quad + \fpi \tint \frac{1}{t + \tau - \tilde t} \ipi (v_1(\tQ)
    \cos 2\theta + v_2(\tQ) \sin 2\theta) \; d\theta d\tilde t
  \end{split}\displaybreak[0] \\
  \label{eq:exactv}
  \begin{split}
    v_2(P) &= \frac{1}{2} v_2(Q_0) + \avgint -\frac{g}{\bar c} h(Q) \sin
    \theta + v_1(Q)\sin\theta \cos\theta+ v_2(Q) \sin^2 \theta \;d\theta \\
    &\quad - \frac{g}{2} \tint h_{x_2}(\tQ_0) + b_{x_2}(\tQ_0) \;
    d\tilde t \\
    &\quad - \frac{g}{2\pi}\tint \ipi (b_{x_1}(\tQ)\cos \theta\sin\theta
    + b_{x_2}(\tQ) \sin^2\theta)\; d\theta d\tilde t  \\
    &\quad + \fpi \tint \frac{1}{t + \tau - \tilde t} \ipi (v_1(\tQ)
    \sin 2\theta + v_2(\tQ) \cos 2\theta) \; d\theta d\tilde t.
  \end{split}
\end{align}
For efficient computations these operators have to be simplified. In
\cite{LukacovaNoelleKraft2007}, the authors follow
\cite{LukacovaMortonWarnecke2004} and present approximations of
(\ref{eq:exacth})~--~(\ref{eq:exactv}) that provide exact solutions of some
one-dimensional Riemann problems. For piecewise constant data, these
approximations read
\begin{align}
\label{eq:foh}
  \begin{split}
  h(P) &=  \avgint \left[ \htot(Q) - \frac{\bar c}{g} \left(v_1(Q)
  \sgn(\cos \theta) + v_2(Q) \sgn(\sin \theta)\right)\right]\; d\theta \\
  &\quad - b(P) + \frac{\tau}{2\pi} \ipi (\bar v_1 b_{x_1}(Q)
  + \bar v_2 b_{x_2}(Q))\; d\theta,
  \end{split}\displaybreak[0] \\
  \label{eq:fou}
  \begin{split}
    v_1(P) &= \avgint \left[-\frac{g}{\bar c} \htot(Q) \sgn( \cos
    \theta)\right.\\ &\quad \phantom{\avgint} \quad \left.+ v_1(Q) 
    \left(\cos^2\theta+\frac12\right)
    + v_2(Q) \sin \theta \cos \theta \right]\; d\theta,
  \end{split}
  \displaybreak[0] \\
  \label{eq:fov}
  \begin{split}
    v_2(P) &= \avgint \left[-\frac{g}{\bar c} \htot(Q)\sgn( \sin
    \theta)\right.\\ &\quad \phantom{\avgint} \quad \left. + v_1(Q) 
    \sin \theta \cos \theta
    + v_2(Q) \left(\sin^2\theta+\frac12\right) \right]\; d\theta.
  \end{split}
\end{align}
The corresponding operators for piecewise (bi-)linear data are given as
\begin{align}
  \label{eq:soh}
  \begin{split}
  h(P) &= \htot(Q_0)(1-\frac{\pi}{2}) -b(P) + \frac14 \ipi \htot(Q)
  \;d\theta\\
  &\quad - \frac{\bar c}{g\pi} \ipi(v_1(Q)  \cos \theta + v_2(Q) \sin \theta)\; d\theta \\
  &\quad + \frac{\tau}{2\pi} \ipi (\bar v_1 b_{x_1}(Q)  + \bar v_2 b_{x_2}(Q))\; d\theta,
  \end{split}\displaybreak[0] \\
  \label{eq:sou}
  \begin{split}
    v_1(P) &= v_1(Q_0)(1-\frac{\pi}{4}) + \frac{g}{\bar c \pi}\ipi  \htot(Q) \cos
    \theta\;d\theta \\
    &\quad +\frac{1}{4}\ipi\left[ v_1(Q)(1+ 3\cos^2\theta) + 3v_2(Q) \sin \theta
      \cos \theta \right] \; d\theta,
  \end{split}\displaybreak[0] \\
  \begin{split}
  \label{eq:sov}
    v_2(P) &= v_2(Q_0)(1-\frac{\pi}{4}) + \frac{g}{\bar c \pi}\ipi  \htot(Q) \sin
    \theta\;d\theta \\
    &\quad +\frac{1}{4}\ipi\left[ 3v_1(Q) \sin \theta
      \cos \theta + v_2(Q)(1+ 3\sin^2\theta) \right] \; d\theta.
  \end{split}
\end{align}
We introduce the operator $\vE(\vW)$ as a shorthand for evaluating
these evolution operators at all quadrature points $\vx_k$ for any given
numerical data $\vW$ defined analogously to (\ref{eq:globalu}).  We
will refer to the operators for piecewise constant data
(\ref{eq:foh})--~(\ref{eq:fov}) as $\ec(P)$ and $\eb(P)$ will denote
the operators for piecewise bilinear data
(\ref{eq:soh})--~(\ref{eq:sov}).

In our schemes, these operators are evaluated at the quadrature points
$\vx_k$ of the finite volume update defined in (\ref{eq:edgeflux}). Thus
all data contributing to the evolved values is derived from the cell
values next to the quadrature point. We therefore define the {\em
  stencil $S_k$} of a quadrature point $\vx_k$ as
\begin{equation}
  \label{eq:evostencil}
  S_k := \{C_i| \vx_k \in \partial C_i\}.
\end{equation}
An example of the intersection of the cone with grid cells and the
resulting stencil is shown in Fig.~\ref{fig:intersect}. 
\begin{figure}
  \centering
  \hfill
  \includegraphics{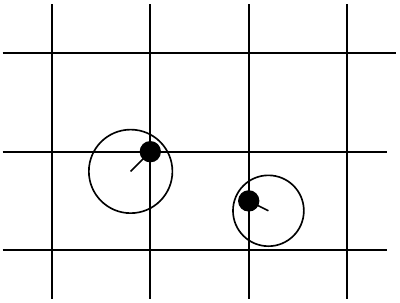} \hfill
  \includegraphics{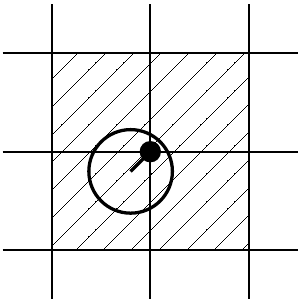} \hspace*{\fill}
  \caption{Left: Intersection of the sonic cone at quadrature points
    with grid cells. Right: Stencil of a quadrature point}
  \label{fig:intersect}
\end{figure}
The suitable average value $\bar \vw_k$ used in (\ref{eq:lSWE}) is
chosen as
\begin{equation}
  \label{eq:avgw}
  \bar \vw_k = \frac{1}{|S_k|} \sum_{i:C_i \in S_k} \vw_i.
\end{equation}
We also tried a local Lax-Friedrichs update at the prediction points
to get a better linearisation. The numerical results were almost
exactly the same, so we chose the averaging procedure (\ref{eq:avgw})
for our computations. 

\subsection{Numerical Representation of the Bottom Topography}
\label{sec:numbot}

For finite volume schemes, the numerical representation of the bottom
topography plays a crucial role in well-balancing as well as
positivity of the scheme. In \cite{AudusseBouchutBristeauEtAl2004}
Audusse \etal{}  use cell averages of the bottom for the
computation of the free surface and reconstruct the free surface and
the water height. The reconstruction of the bottom then results as the
difference between slopes of $\htot$ and $h$. In
\cite{KurganovPetrova2007} Kurganov and Petrova propose to use a
piecewise linear approximation of $b$ instead of $b$ itself by taking
the values of $b$ at cell corners. The cell average of $b$ is then
computed as the average of the corner values.

For the FVEG schemes it is necessary to define some value of $b$ not
only for cell averages and the reconstructed slopes, but also at the
quadrature points where the evolution operators are evaluated.  There
is some freedom in doing this, as the source term discretisation
(\ref{eq:dsource}) respects the well-balancing property independently
of the reconstructed slopes of the bottom topography. As the evolution
operators for the water height compute the free surface first and
derive the actual water height via $h(P) = H(P) - b(P)$, the only
necessary condition for $b(P)$ is $b(P) \leq H(P)$.

In this work, we will define the cell averages of $b$ as in
(\ref{eq:ca}). For the quadrature points on cell corners, we set
\begin{equation}
  \label{eq:avgb}
  b_k :=  \frac{1}{|S_k|} \sum_{i:C_i \in S_k} b_i \approx b(\vx_k).
\end{equation}
The values of $b_k$ at the centres of each edge are linearly
interpolated from the neighbouring corners. While the latter condition
has been derived in \cite{BollermannLukacovaNoelle2009} to ensure
well-balancing on adaptive grids, the formula for the corner points
will turn out to be helpful for the dry bed case.

\subsection{A Multidimensional Entropy Fix for the FVEG scheme}
\label{sec:entropy}

It is well known that the weak solution of a Riemann problem for
conservation laws is not always unique, and an entropy condition is
needed to single out the physically correct solution.  This has its
correspondence on the discrete level, where conservative numerical
schemes may converge to entropy-violating solutions. This notorious
difficulty seems to appear only near sonic rarefaction waves, where
the flow changes from subcritical to supercritical velocity
\cite{Roe1992}. Various researchers have proposed so-called
``entropy-fixes'' for numerical schemes. In particular, we would like
to mention Harten's and Hyman's entropy fix for the the Roe solver
\cite{HartenHyman1983} (see also the discussion in
\cite{LeVeque2002}).
 
\begin{figure}
  \centering
  \hspace*{\fill}
  \includegraphics{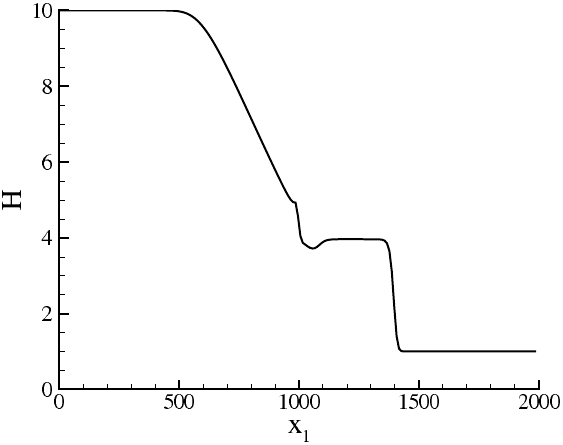}
  \hfill
  \includegraphics{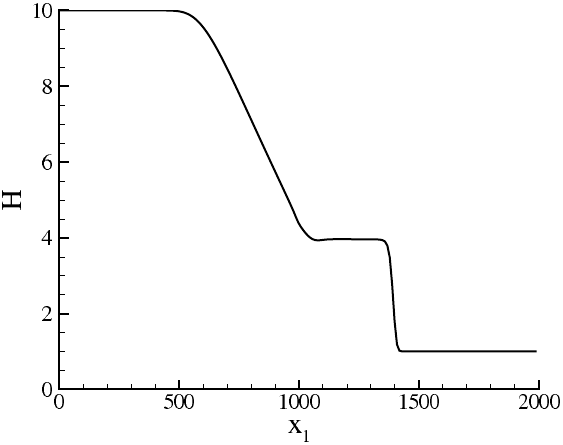}
  \hspace*{\fill}
  \caption{1D dam break problem, solved with first order FVEG
    method. Left: solution without entropy fix. Right: solution with
    entropy fix from \cite{LukacovaTadmor2009}}
  \label{fig:ent}
\end{figure}


The FVEG schemes considered here make no exception, and may compute
entropy violating solutions, see the left picture in
Fig.~\ref{fig:ent}. As is well known for classical finite volume
methods, this effect is less visible (though still there) for second
order schemes \cite{Roe1992}.  In order to make our point clear, we
therefore focus on first order computations for the rest of this
section.

In \cite{LukacovaTadmor2009}, \luka{} and Tadmor proposed an entropy
conservative variant for rarefaction waves computed by certain Riemann
solvers, see also \cite{Tadmor2003}. They applied this technique
successfully to the finite volume corrector step of the FVEG
scheme. They derived just the right amount of viscosity that one
should add to the scheme to fulfil the entropy
equality. Fig.~\ref{fig:ent} clearly shows the effectiveness of the
scheme: While the standard FVEG scheme produces an entropy violating
shock, the entropy conservative scheme clearly reproduces the correct
rarefaction wave. Nevertheless, the scheme from
\cite{LukacovaTadmor2009} does not appear perfectly suitable for our
needs.  First, the proposed fix requires the characteristic
decomposition of the jump of the conserved quantities across an edge.
As the decomposition is not needed for the FVEG schemes, this is an
undesired computational extra cost. In the context of dry boundaries,
we should also mention that the decomposition matrix becomes very ill
conditioned when $h \to 0$. The second point is that the scheme from
\cite{LukacovaTadmor2009} has been developed for the one-dimensional
case. Although it can be applied dimension wise, this approach
somewhat spoils the multidimensional spirit of the FVEG methods.

\begin{figure}
  \centering
  \hspace*{\fill}
  \includegraphics{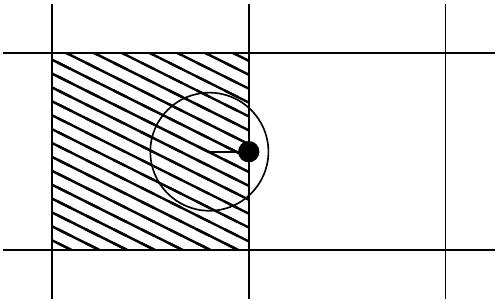}
  \hfill
  \includegraphics{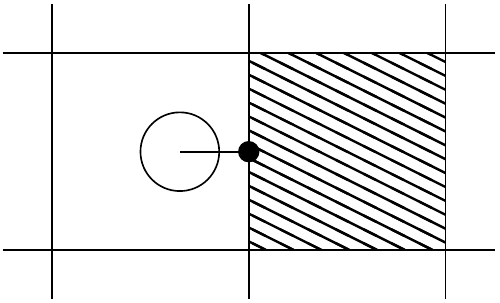}
  \hspace*{\fill}
  \caption{Position of sonic cones for the discrete sonic rarefaction
    with subsonic $\vu_l$ and supersonic $\vu_r$. Left: cone for $v_l
    < c_l$. Right: cone for $v_r>c_r$}
  \label{fig:rf}
\end{figure}

We therefore propose a new approach to solve the entropy problem. It
is not based on a flux correction, but on the correct evaluation of
the EG operators. To motivate our solution, we take a closer look on a
discrete one-dimensional Riemann problem that should result in a
transonic rarefaction, i.e. the flow is subsonic in upwind direction
and supersonic in downwind direction. Thus let us assume we have two
adjacent cells $C_l$ and $C_r$ with cell averaged data
\begin{equation}
  \label{eq:ulur}
  \vu_l = (h_l, v_l, 0),\; v_l < c_l \text{ and } 
  \vu_r = (h_r, v_r, 0),\; v_r > c_r.
\end{equation}
Here $c$ is the speed of sound defined in (\ref{eq:sos}). To evaluate
the evolution operators, we start with the sonic cones defined in
(\ref{eq:Q0})~--~(\ref{eq:tQ}). For simplicity, we limit ourselves to
the quadrature point at the centre of the edge. The sonic cones
resulting from $\vu_l$ and $\vu_r$ are sketched in
Fig.~\ref{fig:rf}. We can see that the cone resulting from $\vu_r$ is
located completely in $C_l$. Depending on the exact values of
$\vu_{l,r}$, this can also be the case for the sonic cone resulting
from the averaging procedure (\ref{eq:avgw}). In other words: We use
an evolution operator resulting from a supersonic linearisation in a
regime that is subsonic. At least for the first order operators
(\ref{eq:foh})~--~(\ref{eq:fov}), this means that the predictor step
exactly reproduces $\vu_l$, which is then used for the flux
evaluation. This corresponds to the generalised upwind method which is
known to compute entropy violating solutions in some situations,
cf. e.g. \cite{LeVeque2002}.

\begin{figure}
  \centering
  \hspace*{\fill}
  \includegraphics{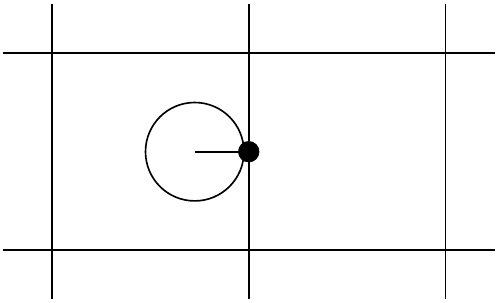}\hfill
  \includegraphics{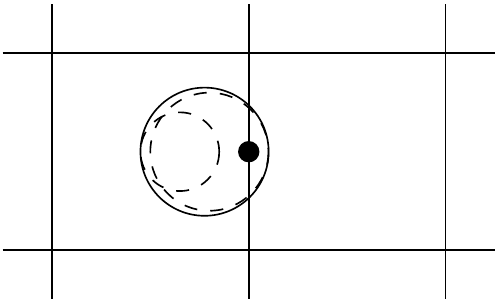}
  \hspace*{\fill}
  \caption{Position of sonic cones for the discrete sonic rarefaction
    with subsonic $\vu_l$ and supersonic $\vu_r$. Left: cone for
    averaged value $\bar \vu$. Right: superscribed cone.}
  \label{fig:rf_new}
\end{figure}

Thus the core of the problems seems to be the wrong domain of
dependence for the predictor step. In case of a sonic rarefaction, the
sonic cone should always include both regions, the subsonic as well as
the supersonic one. As this is not guaranteed, we modify our method by
extending the sonic cone if necessary. In a transonic situation we drop
the sonic circle resulting from the averaging procedure
(\ref{eq:avgw}). Instead, we use a circle which comprehends all the
circles defined by the cell averaged values in the corresponding
stencils, see Fig.~\ref{fig:rf_new} for an illustration. The exact
formulation used for our schemes is as follows. Given two circles with
midpoints $\vx_i$ and radii $r_i$, $i=1,2$, we compute the new circle
as
\begin{align*}
  r &= \frac{r_1 + r_2 + d}{2} \\
  \vx &= \vx_1 + (r - r_1)\frac{\vx_2 - \vx_1}{d}
\end{align*}
with $d=\|\vx_2 - \vx_1\|_2$. However, if one circle comprehends the
other one, we just choose the bigger circle as our new one. If the
stencil of the evolution point consists of four cells, we apply the
same formula for the neighbours on the two diagonals first and then
again for the resulting circles.

\begin{figure}
  \centering
  \hspace*{\fill}
  \includegraphics{dam_LT.pdf}
  \hspace*{\fill}
  \includegraphics{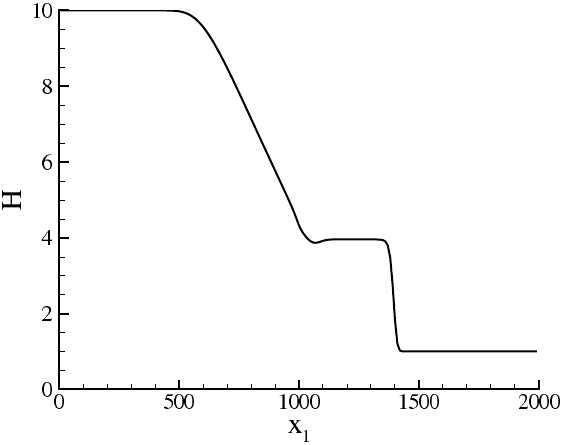}
  \hspace*{\fill}
  \caption{1D dam break problem, solved with first order FVEG
    method. Left: solution with entropy fix from
    \cite{LukacovaTadmor2009} Right: solution with new entropy fix}
  \label{fig:ent_comp}
\end{figure}


In Fig.~\ref{fig:ent_comp} we compare the results of the entropy
stable scheme from \cite{LukacovaTadmor2009} and our new
approach. They both clearly solve the entropy problem, with a slight
advantage of the scheme by \luka{} and Tadmor. However, as we
previously pointed out, the new method is more efficient. Compared
to the scheme without any entropy fix, the new one took only about 2\%
extra time, whereas the approach from \cite{LukacovaTadmor2009} needs about
8\% more computational time. As the new scheme is also suitable for
computations including dry areas, we chose it for all the numerical
experiments in Section~\ref{sec:res}.

\section{Dry Bed Modifications}
\label{sec:dry}

To extend our schemes to computations including dry beds, we have to
guarantee two properties: the positivity of the water height, and the
well-balancing under the presence of dry areas. In literature, this is
mainly achieved by two basic ingredients: a positivity preserving
reconstruction and an additional time step constraint. Examples can be
found in \cite{AudusseBouchutBristeauEtAl2004, KurganovPetrova2007,
  RicchiutoBollermann2009}.

For the FVEG schemes, these measures fall short of the aims. The
additional predictor step via the evolution operator prevents a direct
proof of a positivity property. One reason for this is the extended
stencil: the flux over an edge is computed using more cells than the
direct neighbours (see Fig.~\ref{fig:quad} and \ref{fig:intersect}).
Another problem is the complex evaluation of the operators and with
them the flux which makes an analysis of the positivity at least
challenging if not impossible.

Regarding the well-balancing, a sophisticated reconstruction is not
enough. From (\ref{eq:dsource}) it is obvious that the reconstruction
does not directly affect the balancing of flux and source terms. The
core of the problem is that the lake at rest described in
(\ref{eq:lar}) changes to
\begin{equation}
  \label{eq:lardry}
     \vv = (0,0)^T \text{ and } \htot(\vec x) =
     \begin{cases}
       H_0 & \text{$h(\vx) > 0$}\\
       b(\vx) & \text{else}
     \end{cases}
\end{equation}
if dry areas are included. Thus for our schemes, we have to find
evolution values for the water and bottom height that can handle
properly the occurrence of this situation, i.e. which avoid the
generation of spurious waves at the shoreline.

In this section, we will present an alternative approach to ensure the
positivity of the water height as well as modifications of the finite
volume update and the evolution operator to ensure the well-balancing
property. We make sure that the changes do not affect the scheme away
from dry regions.


\subsection{A General Positivity Preserving FV Update}
\label{sec:posFV}

For the derivation of a positivity preserving scheme, we study
the first component of the finite volume update
(\ref{eq:fullydiscrete})
\begin{equation}
\label{eq:fvh}
  h_i^{n+1} = h_i^n - \frac{\dt}{\dx} \sum_{E, E \subset \partial C_i}
  \hflux_E
\end{equation}
where
\begin{equation}
\label{eq:fluxh}
\hflux_E :=  \flux_E^h
\end{equation}
is the first component of the flux vector $\flux_E$. In the following,
we will modify the flux $\hflux_E$ in order to guarantee positive water
height. The technique presented here is applicable to arbitrary finite
volume fluxes, and the FVEG flux given in (\ref{eq:edgeflux}) below is
only a
special case within this framework.

The basic idea of our method is to cut off the outgoing fluxes as soon as
all water which has been contained in the cell at the beginning of the
time step has left the cell via the outgoing fluxes. We will call this
time the {\em draining} time. For convenience, we will later rewrite
this as a reduced time step $\dt_E$ on these edges, but in fact the 
finite volume update will always advance the solution by one 
global time step $\dt$.

Thus our first step is to separate the fluxes contributing to the
outflow off a cell $C$ from those contributing to the inflow by setting
\begin{align}
  \label{eq:fin}
  \hflux^+_E &:= \max\left\{\hflux_E, 0\right\} \\
  \hflux^-_E &:= \min\left\{\hflux_E, 0\right\}.
\end{align} 
This allows us to rewrite the update (\ref{eq:fvh}) as
\begin{equation}
  \label{eq:hinout}
    h_i^{n+1} = h_i^n \;-\; \underbrace{\frac{\dt}{\dx} \sum_{E, E \subset
     \partial C_i} \hflux_E^+}_{\text{outflow}}  \;-\; \underbrace{\frac{\dt}{\dx}
      \sum_{E, E \subset \partial C_i} \hflux_E^-}_{\text{inflow}}.
\end{equation}
Now we introduce the {\em draining} time by
\begin{equation}
  \label{eq:dtdrain}
  \dtd := \frac{\dx \, h_i^n}{\sum_{E, E \subset \partial C_i}  \hflux_E^+ }.
\end{equation}
Once again, at time $t^n+\dtd$ all water which was originally contained
in cell $C_i$ has flown out, so
\begin{equation}
  \label{eq:hdrain}
    h_i^n \;-\; \underbrace{\frac{\dtd}{\dx} \sum_{E, E \subset \partial
      C_i} \hflux_E^+}_{\text{outflow}} = 0.
\end{equation}
Suppose now that $\dtd<\dt$. For $t\in[t^n+\dtd\,,\,t^{n+1}]$, no more
water can leave the cell, at least not water which was originally
contained in cell $C_i$. Therefore we assume that there is no outgoing
flux for times beyond the draining time, and introduce the cut-off
flux $\widetilde\hflux_E^+$ by
\begin{equation}
  \label{eq:fdrain}
  \widetilde\hflux_E^+(t) := \begin{cases} \hflux_E^+(\vu^n) & \text{ for} \;\;
  t^n\;\leq\;t\;<\;t^n+\dtd
   \\ 
  \quad 0 & \text{ for} \;\; t\;>\;t^n+\dtd \end{cases}
\end{equation}
Now we integrate the draining flux in time and obtain the cut-off (or draining)
finite volume flux
\begin{equation}
  \label{eq:hfluxdrain}
  \dt \, \hfluxpd(\vu^n)  := \int_{t^n}^{t^{n+1}} \widetilde\hflux_E^+(t) dt
  = \int_{t^n}^{t^{n}+\dtd} \hflux_E^+(\vu^n) dt
  = \dtd \hflux_E^+(\vu^n).
\end{equation}
Before introducing the positivity-preserving modification of the
finite volume scheme \eqref{eq:fullydiscrete}, we rewrite the cut-off in the
flux as a local cut-off in the time step. This cut-off time step
is defined for each edge $E$ and takes into account
the upwind cell $C^-(E)$:
\begin{equation}
  \label{eq:dte}
  \dt_E := \min\left(\dt, \dt_{C^-(E),\text{drain}}\right)
\end{equation}
Now we replace the finite volume flux $\hflux$ in
\eqref{eq:fullydiscrete} by the cut-off finite volume flux
defined in \eqref{eq:hfluxdrain}.
Using \eqref{eq:dte}, this leads to the modified general
update (\ref{eq:fullydiscrete})
\begin{equation}
  \label{eq:fvlts}
  \vu_i^{n+1} = \vu_i^n -  \frac{1}{\Delta x}  \sum_{E, E
    \subset \partial C_i} \dt_E\flux_E.
\end{equation}

\begin{theorem}
  Assume we have a conservative finite volume scheme for the solution
  of the shallow water equations that can be written in the form
  (\ref{eq:fullydiscrete}). Then the modified finite volume scheme
  \eqref{eq:fvlts} with locally cut-off flux \eqref{eq:hfluxdrain})
  (respectively locally cut-off time step \eqref{eq:dte})
  is positivity preserving.
\end{theorem}

\begin{proof}
From the first component of \eqref{eq:fvlts}, combined with definitions \eqref{eq:dte}
of the cut-off time step $\dt_E$ and \eqref{eq:dtdrain} of the draining time $\dtd$,
the water height at the new time step can be bounded as follows:
\begin{align*}
  \label{eq:hinout}
    h_i^{n+1} 
    &\;=\;\;\; h_i^n 
     \;\;-\; \sum_{E, E \subset \partial C_i} 
       \frac{\dt_E}{\dx} \; \left(\hflux_E^+ + \hflux_E^-\right) \\
    &\;\geq\;\;\; h_i^n 
     \;\;-\; \sum_{E, E \subset\partial C_i}
       \frac{\dt_E}{\dx} \; \hflux_E^+ \\
    &\;\geq\;\;\; h_i^n 
     \;\;-\; \sum_{E, E \subset\partial C_i}
       \frac{\dt_{C^-(E),\text{drain}}}{\dx} \hflux_E^+ \\
    &\;=\;\;\; h_i^n 
     \;\;-\; \frac{\dt_{C_i,\text{drain}}}{\dx} 
       \sum_{E, E \subset\partial C_i} \hflux_E^+ \\
    &\;=\;\;\; 0
\end{align*}

\vspace*{-3ex}

\qedhere
\end{proof}

\begin{remark}

The time step $\dt_E$ used in the new finite volume update \eqref{eq:fvlts} and
defined in \eqref{eq:dte} might seem to be local for each edge. We would like to
stress, however, that the finite volume scheme  \eqref{eq:fvlts} still advances the
solution by one and the same global time step $\dt$. The apparent contradiction
is resolved by considering equation \eqref{eq:hfluxdrain}: the time-integral
of the flux is still over the global interval $[t^n,t^{n+1}]$. However, the
flux $\hfluxpd$ is cut-off in the presence of vacuum, see \eqref{eq:fdrain}.

\end{remark}

\subsection{Well-balancing at the Shoreline: the Finite Volume Update}
\label{sec:wbsource}

In the derivation of the positivity preserving finite volume update,
we so far neglected the source term. Its introduction to the new
scheme (\ref{eq:fvlts}) rises the question which time step should be
used for the source term. To maintain the well-balancing, the source
term and the gravity driven parts of the flux must be multiplied
with the same time step. This is in contradiction to the definition of
$\dt_E$, which may change for different edges of the same cell. On the
other hand, the reduced time step is not necessary for the momentum
equations. We therefore shift the gravity driven components of $\flux$
into the source term, i.e. we define
\begin{equation}
  \label{eq:newflux}
  \flus(\vu) := 
  \begin{pmatrix}
    hv_1 & hv_2 \\
    hv_1^2 & hv_1v_2 \\
    hv_1v_2 & hv_2^2 
  \end{pmatrix} \text{ and }
  \sours(\vu,\vx) :=  gh
  \begin{pmatrix}
    0 \\ \frac{\partial \htot(\vec x)}{\partial x_1} \\ 
    \frac{\partial \htot(\vec x)}{\partial x_2}
  \end{pmatrix}.
\end{equation}
By replacing $\flux$ with $\flus$ in (\ref{eq:edgeflux}) and changing
(\ref{eq:dsource}) to
\begin{equation}
  \label{eq:discsourcewb}
  \sours_i := g  \sum_{j=1}^3 \alpha_j
  \begin{pmatrix}
      0 \\ \frac12(\hat h_j^r + \hat h_j^l)(\htot^r_j-\htot_j^l) \\
      \frac12(\hat h_j^t + \hat h_j^b)(\htot_j^t-\htot_j^b) 
    \end{pmatrix}
\end{equation}
we can rewrite (\ref{eq:fullydiscrete}) as
\begin{equation}
  \label{eq:newfvupd}
  \vu_i^{n+1} = \vu_i^n -  \frac{1}{\Delta x} \left[\left( \sum_{E, E
    \subset \partial C_i} \dt_E \flus_E\right) + \dt\sours_i \right].  
\end{equation}
This formulation ensures the well-balancing of the scheme even in the
case of modified local time steps. Away from the shoreline we have
$\dt_E = \dt \; \forall E$ and (\ref{eq:newfvupd}) equals the
original update (\ref{eq:fullydiscrete}).

\begin{remark}
  The rearrangement of the advective and gravity driven parts of the
  equations in (\ref{eq:newflux}) is independent of the scheme. Thus
  in principle, every finite volume scheme can be reformulated as in
  (\ref{eq:newfvupd}). To obtain a well-balanced scheme, we only need
  a discretisation of $\sours$ analogously to (\ref{eq:discsourcewb})
  that preserves the lake at rest solution in the presence of dry
  boundaries.
\end{remark}

\newpage

\subsection{Well-balancing at the Shoreline: the Evolution Operator}
\label{sec:wbeg}

The lake at rest situation with dry beds described in
(\ref{eq:lardry}) is only preserved if the numerical flux and source
terms in (\ref{eq:discsourcewb}) are exactly balanced, or equivalently
if the evolution operators reproduce the lake at rest. This is not
necessarily the case if the stencil of a quadrature point contains dry
cells, as is demonstrated in Fig.~\ref{fig:lar}.  If $b_i > H_o$ for a
cell in the stencil, the resulting bottom value from (\ref{eq:avgb})
can be higher than the free surface in the wet cells. Using the
approximate evolution operators (\ref{eq:foh}) and (\ref{eq:soh}), it
is easy to see that in the lake at rest case the evolved water height
is also positive, leading to an even higher free surface at the
quadrature point.
\begin{figure}
  \centering
   \includegraphics{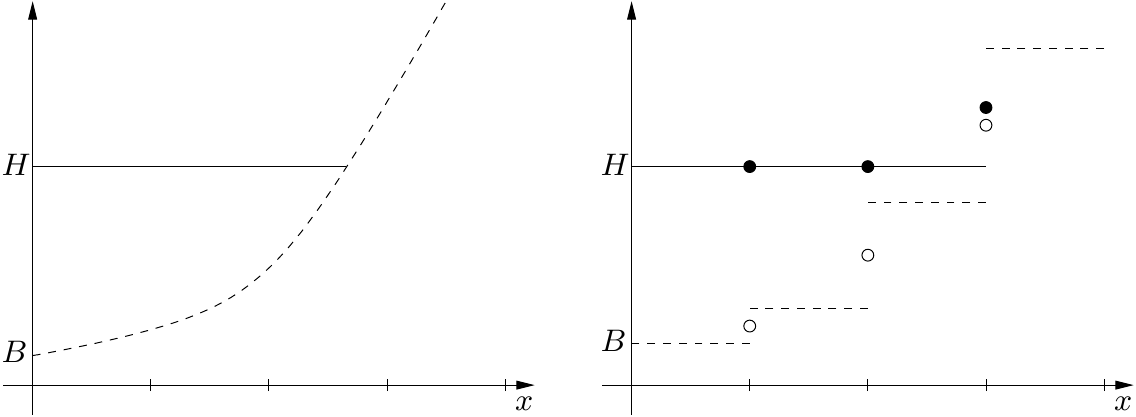}
  \caption{Lake at rest with dry boundaries. Solid line: free surface,
    dashed line: bottom topography, filled circles: $H$ at evolution
    points, empty circles: $B$ at evolution points.  Left: Real
    situation. Right: Numerical Representation with evolution
    values.}
  \label{fig:lar}
\end{figure}
In this case the combined flux and source term $\sours$ from
(\ref{eq:discsourcewb}) does not vanish anymore and introduces
unphysical waves starting from the dry boundary.

To avoid the creation of these waves, we modify the data used for the
predictor step at the interface. First, we replace the stencil $S_k$
by $S^*_k$ defined as
\begin{equation}
  \label{eq:drystencil}
  S^*_k := \{C_i| C_i \in S_k \wedge h_i > 0\}
\end{equation}
which allows us to determine the maximal free surface level at $\vx_k$
as
\begin{equation}
  \label{eq:fsxk}
  \bar H_k = \max_{S^*_k}(H_i).
\end{equation}
Now in (\ref{eq:avgw}) and for the evaluation of the evolution
operators we set
\begin{equation}
\label{eq:dryvals}
(H, b, \vv)_i = (\bar H_k, \bar H_k, 0) \quad \text{ if } C_i \notin
S^*_k \wedge b_i > \bar H_k.
\end{equation}
In all other cases, we leave the values unchanged. The modification,
which is illustrated in Fig.~\ref{fig:larmod}, ensures that the free
surface is correctly represented in the source term computation
(\ref{eq:discsourcewb}). We also avoid an unphysical flooding of
mounting slopes. In \cite{BrufauGarcia-Navarro2003,
  RicchiutoBollermann2009} a similar technique is used on triangles.
\begin{figure}
  \centering
   \includegraphics{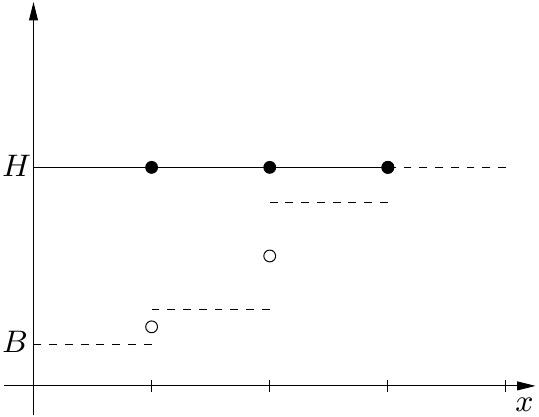}
   \caption{Lake at rest with dry boundaries, modification
     (\ref{eq:dryvals}) for computation of evolution values. Symbols
     like in Fig.~\ref{fig:lar}}
  \label{fig:larmod}
\end{figure}

Finally, even in the presence of wet, but nearly dry cells in the
stencil, the expressions for $h(P)$ in (\ref{eq:foh}) and
(\ref{eq:soh}) can become negative if $h_i$ is small in the
surrounding cells. This cannot necessarily be cured by a smaller
time step, as substantial parts of the expressions do not depend on
$\dt$. With this restriction in mind, we propose the simplest
solution: Whenever we have $h(P) < 0$, we set $h(P) = v_1(P) = v_2(P) =
0$.

\subsection{Treatment of Nearly Dry Cells}
\label{sec:drycells}

In our schemes, we consider a cell $C_i$ to be dry when $h_i <
\eps_H$, where we have chosen $\eps_h = 10^{-8}$. In dry cells we set
$h_i = u_i = v_i = 0$. A well known problem occurs when $h_i$ is close
to that value: the velocity $v = hv/h$ can become singular due to
small numerical errors in the conserved variables. This leads to very
small time steps which in the end can basically stop the
computation. This problem has been discussed e.g. in
\cite{KurganovPetrova2007, RicchiutoBollermann2009}, where different
strategies have been proposed. In \cite{KurganovPetrova2007}, Kurganov
and Petrova propose to de-singularise $v$ by multiplying it with a
certain factor $f < 1$ whenever $h$ falls below a certain threshold
$\eps_v$.  In \cite{RicchiutoBollermann2009}, the authors just set
$v=0$ whenever $h < \eps_v$.

In this work, we use a different approach. As the solution at the dry
boundary is always a rarefaction wave, the flow velocity will grow
smoothly when water floods formerly dry areas. We will therefore limit
the velocity in nearly dry regions depending on the velocity in
flooded areas. We define the reference speed 
\begin{equation}
\label{eq:uref}
v_{ref} := \max_{i: h_i > \eps_v}\|\vv\|_2
\end{equation}
where
\begin{equation}
  \label{eq:epsu}
  \eps_v = \frac{\dx}{L_{ref}}, \quad L_{ref}:= \max_{i,j}\|\vx_i-\vx_j\|_\infty.
\end{equation}
Whenever we have $h_i < \eps_v$ and $\|\vv_i\| > v_{ref}$, we
set the new velocity to
\begin{equation}
  \label{eq:unew}
  v^*_i = v_{ref}\left(2 - \frac{v_{ref}}{\|\vv_i\|}\right),
\end{equation}
such that $\|\vv\|$ is smoothly limited to a value between $v_{ref}$ and
$2v_{ref}$.  The velocity components in $C_i$ are then defined as
\begin{equation}
  \label{eq:uvnew}
  \vv_i = v^*_i  \vec d
\end{equation}
with $\vec d$ the unit vector pointing in the same direction as the
vector of discharge $(hv_1, hv_2)^T$. This approach appears us to be a
better representation of the physics of the flow, as the velocity at
the front is not necessarily vanishing.

\subsection{The FVEG Algorithm}
\label{sec:alg}

Before summarising the whole FVEG algorithm, we will spend a few words
on the reconstruction needed to evaluate the evolution operators for
piecewise linear data (\ref{eq:soh})~--~(\ref{eq:sov}). As the
operators are computed from the primitive variables $\vw$, these are a
natural choice for the reconstruction $R_{\dx}$. Thus in each cell we need
the linear function
\begin{equation}
  \label{eq:recw}
   R_{\dx}(\vw_i)(\vx)_{|C_i} := \tilde \vw_i + \nabla\vw_{i}\cdot(\vx - \vx_i) +
   (\vw_{i})_{x_1 x_2}(\vx-\vx_i)_1(\vx-\vx_i)_2.
\end{equation}
The derivatives $(\vw_i)_{x_1},(\vw_i)_{x_2}$ and $(\vw_i)_{x_1x_2}$
are computed from the slopes between cell averages, cf.
\cite{LukacovaSaibertovaWarnecke2002}. In this paper, we use the
continuous, piecewise bilinear recovery described in
\cite{LukacovaNoelleKraft2007} without any limiters. The piecewise
bilinear functions are uniquely defined by the averages at the cell
corners, which are already computed for the evaluation of the
evolution operators, see (\ref{eq:avgw}). Then the $\tilde \vw$ from
(\ref{eq:recw}) is exactly the average of the averages at the cell
corners, thus the resulting reconstruction is not necessarily
conservative.  We will therefore use the combined evolution operator
\begin{equation}
  \label{eq:combe}
  \vE(\vW) := \eb(R_{\dx}(\vW)) + \ec(\vW - \tilde \vW),
\end{equation}
where the first order correction $\ec(\vW - \tilde \vW)$ is necessary
for stability, see \cite{LukacovaMortonWarnecke2004} for an in-depth
discussion of this issue.

Although the conservative correction introduces some oscillations at
steep fronts, the piecewise bilinear reconstruction has a clear
advantage in the vicinity of dry areas. As the averaged water height
at the cell corners is non-negative via
eqs. (\ref{eq:drystencil})~--~(\ref{eq:dryvals}), the reconstruction
is also non-negative by design. In dry cells, we set
\begin{equation}
  \label{eq:recdry}
  \vw_{x_1} = \vw_{x_2} = \vw_{x_1x_2} = 0  \text{ if } h_i =0.
\end{equation}
We refer to \cite{LukacovaNoelleKraft2007,
  LukacovaSaibertovaWarnecke2002} for further details concerning the
reconstruction strategy.

The complete algorithm now reads as follows:
\begin{alg}
  \begin{enumerate}
  \item From given conservative data $\vu_i^n$ and $b_i$ at time
    $t^n$, compute the nonconservative variables $H_i^n, v_{1,i}^n,
    v^n_{2,i}$.
  \item apply the reconstruction operator $R_{\dx}$ to $H_i^n, v_{1,i}^n,
    v^n_{2,i}$ and $b_i$.
  \item compute the evolution operators
  \item evaluate the advection fluxes $\flus_E$ from
    (\ref{eq:newflux})
  \item compute the gravity driven flux and source terms $\sours_i$
    from (\ref{eq:discsourcewb})
  \item perform the finite volume update (\ref{eq:newfvupd})
  \end{enumerate}
\end{alg}
We will finish the section by a proof of the well-balancing property
of the scheme.
\begin{theorem}
  Suppose that we have a numerical solution respecting the lake at
  rest solution (\ref{eq:lardry}) with dry boundaries. Then the FVEG
  scheme (\ref{eq:newfvupd}) together with the modifications described
  in Section~\ref{sec:wbeg} preserves this state.
\end{theorem}
\begin{proof}
  For the lake at rest state with $\vv = (0,0)^T$, the advective parts
  of the fluxes $\flux$ defined in (\ref{eq:watervariables}) and
  $\flus$ defined in (\ref{eq:newflux}) are all zero. Thus for all
  edges we have $\dt_E = \dt$ and the original finite volume update
  (\ref{eq:fullydiscrete}) and the modified update (\ref{eq:newfvupd})
  are the same. Then Theorem 2.1 from \cite{LukacovaNoelleKraft2007}
  states that the scheme is well-balanced provided that the predicted
  point values used for the flux evaluation also satisfy the lake at
  rest situation.

  We will now show that all data used for the reconstruction and for
  the evaluation of the predictor step satisfies the requirements of
  Theorem 3.1 from \cite{LukacovaNoelleKraft2007}. From definitions
  (\ref{eq:drystencil}) and (\ref{eq:fsxk}) we see that for all
  evolution points the averaged free surface is computed as $H_0$,
  which is the free surface level in all flooded cells. As all
  velocities are zero by (\ref{eq:lardry}) respectively
  (\ref{eq:dryvals}), the averaging also returns zero. Finally, the
  reconstruction procedure is based on the averaged point values and
  therefore in all flooded cells we have $ \vw_{x_1} = \vw_{x_2} =
  \vw_{x_1x_2} = 0$. In dry cells we have the same result by
  definition (\ref{eq:recdry}). Thus we can apply Theorem 3.1 from
  \cite{LukacovaNoelleKraft2007} and this concludes our proof.
\end{proof}

\section{Numerical Results}
\label{sec:res}

\subsection{Dam Break over Dry Bed}
\label{sec:dam}

\begin{figure}
  \centering
  \hspace*{\fill}
  \includegraphics{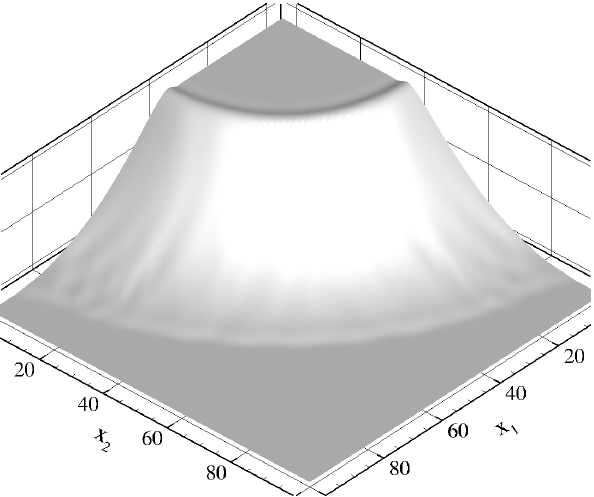} \hfill
   \includegraphics{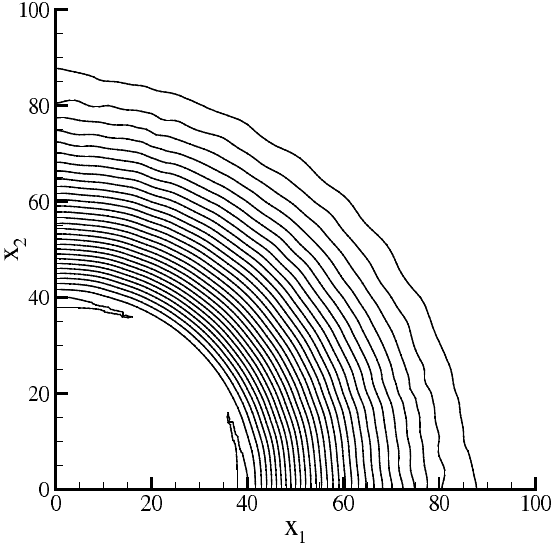}
  \hspace*{\fill}
  \caption{Circular dam break over dry bed, solution at t=1.75s. Left:
    3D view. Right: 30 contour lines between $H=10.2$ and $H=0.2$}
  \label{fig:dam1}
\end{figure}

This is a classical test case where we simulate the complete break of
a circular dam separating a basin filled with water from a dry
area. The computational domain is $[0,100]^2$ and we set $\dx = 1$,
the water filled basin is located at $r = \|\vx\| \leq 60$. In the
basin we set $H_0 = 10$ and elsewhere $H_0=0$ and the initial velocity
is $\vv_0 = (0,0)^{tr}$ in the whole domain. Reference solutions can be found
in \cite{RicchiutoBollermann2009, Seaid2004,
  AlcrudoGarcia-Navarro1993}.

In Fig.~\ref{fig:dam1}, we see a 3D view and contour lines of the
water height, Fig.~\ref{fig:dam2} shows the water height and velocity
at different lines through the domain. The resulting rarefaction wave
is almost perfectly symmetric, the oscillation due to the
reconstruction strategy is restricted to three percent of the water
height (we have $\max_i H_i = 10.269$). We see a small bump at the
drying wetting front, but the front position and velocities are well
represented. Thanks to the new entropy fix, there is no unphysical
shock visible in the transcritical region.

\begin{figure}
  \centering
  \hspace*{\fill}
  \includegraphics{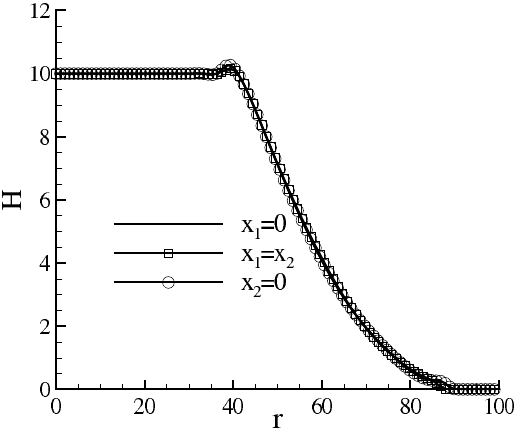}
  \hfill
  \includegraphics{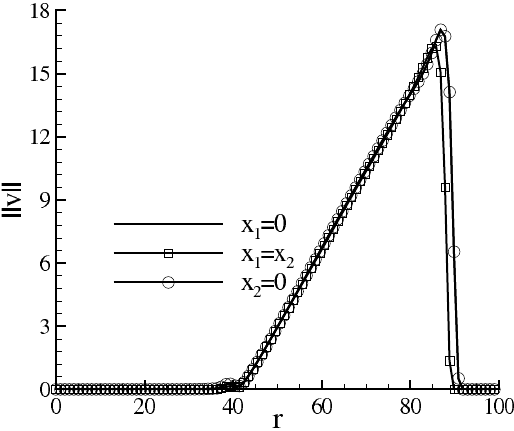}
  \hspace*{\fill}
  \caption{Circular dam break over dry bed, solution at t=1.75s with
    $r=\|\vx\|_2$. Left: free surface. Right: velocity}
  \label{fig:dam2}
\end{figure}

\subsection{Wetting/Drying on a Sloping Shore}
\label{sec:slope}

This test case was proposed by Synolakis in \cite{Synolakis1987} and
computed in e.g. \cite{RicchiutoBollermann2009,
  MarcheBonnetonFabrieEtAl2007}. It describes the run-up and reflection
of a wave on a mounting slope, with the initial solution given as
\begin{equation}
  \label{eq:synolakis}
  H_0(\vx) = \max(f, b(\vx)), \quad \vv_0(\vx) =
  \left(\sqrt{\frac{g}{D}} H_0(\vec x), 0 \right)^T
\end{equation}
where
\begin{equation}
  \label{eq:synf}
  f(\vx) = D + \delta \sech^2(\gamma(x_1 - x_a)).
\end{equation}
As in \cite{RicchiutoBollermann2009, MarcheBonnetonFabrieEtAl2007}, we
set $D=1, \delta = 0.019$ and
\begin{equation}
  \label{eq:synpar}
  \gamma = \sqrt{\frac{3\delta}{4D}}, \qquad x_a =
  \sqrt{\frac{4D}{3\delta}} \mathrm{arcosh}\left(\sqrt{20} \right).
\end{equation}
For the bottom topography we have
\begin{equation}
  \label{eq:synb}
  b(\vx) = b(x_1) =
  \begin{cases}
    0 & \text{$x_1< 2x_a$} \\
    \dfrac{x_1 - 2x_a}{19.85} & \text{else.}
  \end{cases}
\end{equation}
The computational domain is $\Omega = [0,80]\times[0,2]$ and the grid
size $\dx = 0.04$. We prescribe open boundary conditions in $x_1$
direction and periodic ones in $x_2$ direction.

\begin{figure}
  \centering
  \hspace*{\fill}
  \includegraphics{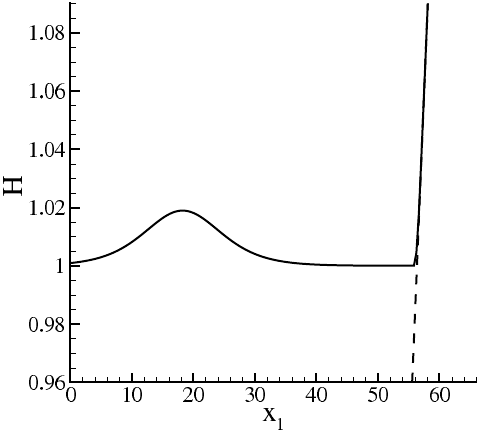} \hfill
  \includegraphics{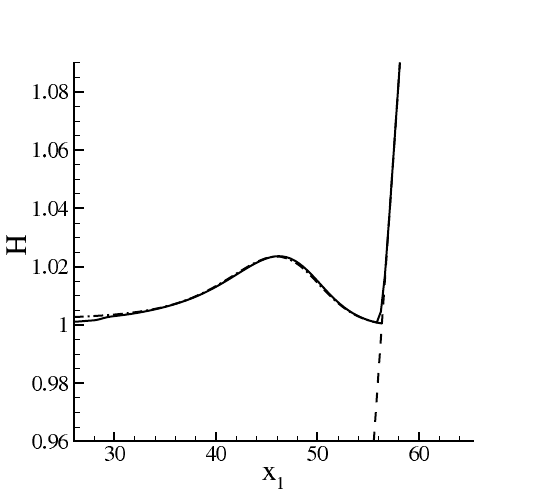}   \hspace*{\fill} 

  \vspace{5mm}

  \hspace*{\fill}
  \includegraphics{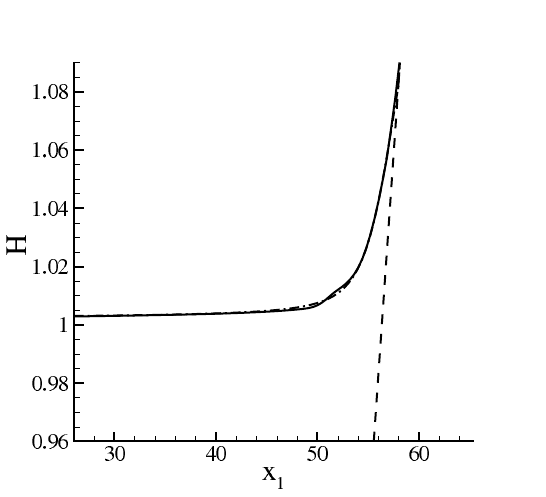} \hfill
  \includegraphics{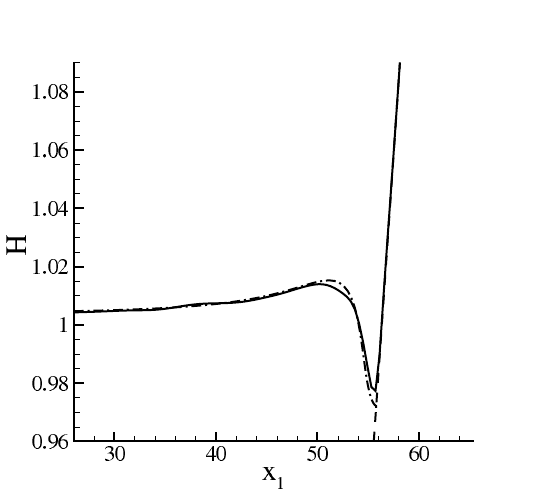}
  \hspace*{\fill}

  \vspace{5mm}

  \hspace*{\fill}
  \includegraphics{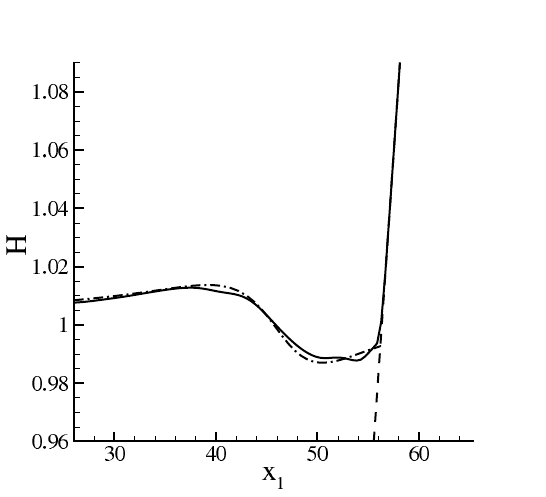}\hfill
  \includegraphics{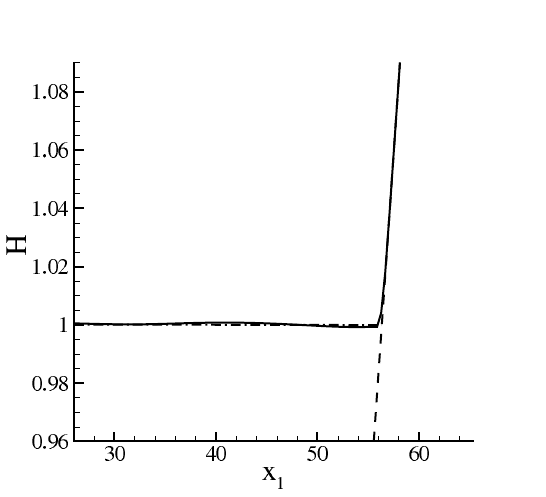}
  \hspace*{\fill}
  \caption{Drying/wetting on a sloping shore, free surface. Solid
    line: numerical results. Dashed-dotted line: Analytical result
    described in \cite{Synolakis1987}. Dashed line: bottom
    height. From top left to bottom right:
    Solutions at times $t=0$, $t=9$, $t=17$, $t=23$, $t=28$ and
    $t=80$.}
  \label{fig:syn}
\end{figure}

In Fig.~\ref{fig:syn} we present the water height during the run-up
and drying process together with the analytical solution. Details on
how to obtain the analytical solution can be found in \cite[Section
3.5.2]{Synolakis1986}. At time $t=9$, the wave has almost reached the
shoreline, and shortly after, at $t=17$, the wave reaches the maximal
run-up on the shore and the drying process starts. During the run-up,
the agreement with the analytical solution is excellent. The drying
process is reproduced very accurate, too, although compared to the
run-up process, there are small deviations between numerical and
analytical solution with respect to the minimum of the water height at
$t=23$. At $t=28$, where the reflected wave starts leaving the domain,
these deviations persist, but stay very small. During the whole
simulation, no oscillations or other perturbations at the
wetting/drying front are visible, demonstrating the well-balancing
capabilities of the scheme. The scheme also returns quickly to the
lake at rest solution presented in the last picture ($t=80$). where
the water surface is almost flat. All in all, the results are very
satisfying and compare well to the other numerical solutions presented
in \cite{RicchiutoBollermann2009, MarcheBonnetonFabrieEtAl2007}.

\subsection{Vacuum occurrence by a double rarefaction wave over a step}
\label{sec:vacuum}

To test how the scheme handles the drying of formerly flooded areas,
we compute a test case proposed in \cite{GallouetHerardSeguin2003}. It
describes two separating waves over a non-flat bottom. The
computational domain is a pseudo-1D channel given as $\Omega = [0,
25]\times [0, 0.5]$. We set the initial free surface height to $H^0
= 10$ and the discharge and bottom topography to
\begin{equation}
  \label{eq:moses_init}
    hv_1(\vx, 0) =
    \begin{cases}
      350 & \text{if $x_1 > 50/3$} \\
      -350 & \text{else}
    \end{cases} \qquad
    b(\vx) =
    \begin{cases}
      1 & \text{if $25/3 < x_1 < 25/2$} \\
      0 & \text{else}
    \end{cases}.
\end{equation}
Like in \cite{GallouetHerardSeguin2003}, the computation is performed
on a grid with 300 grid cells in $x_1$ direction.

In Fig.~\ref{fig:vac}, we show the free surface and the discharge of
the solution at different times $t$. At time $t=0.5$, several waves
have arisen from the interaction of the supersonic flow with the
bottom topography. Due to the reconstruction strategy described in
Sec.~\ref{sec:alg}, the rarefaction waves are smoothed out at the
bottom and show small oscillations at the top. However, for the
following times we see an accurate representation of the waves flowing
over the hump ($t=0.25$) and leaving the domain ($t=0.45,0.65$). No
spurious modes are introduced during the drying process, neither in the
flat region nor at the edges of the hump.
\begin{figure}
  \centering
  \hspace*{\fill}
  \includegraphics{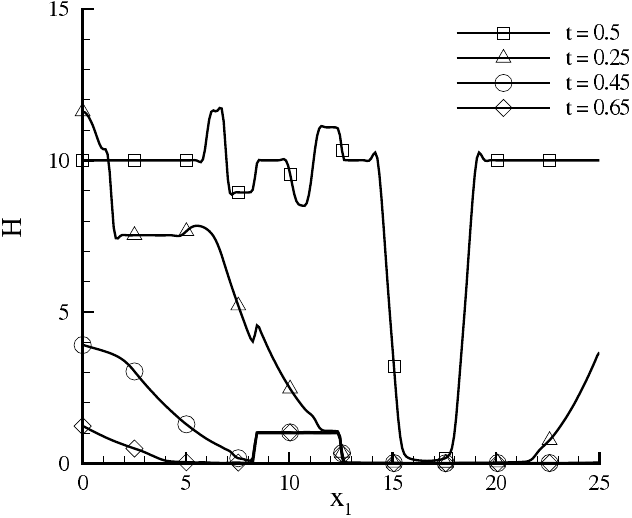}
  \hfill
   \includegraphics{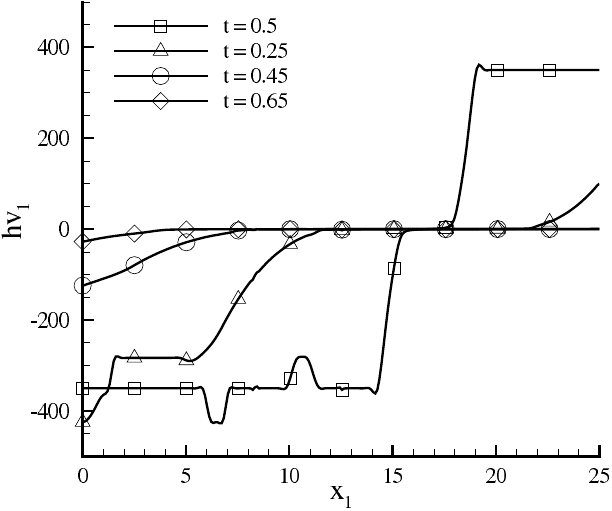}
  \hspace*{\fill}
  \caption{Vacuum occurrence, solution at different times. Left:
    free surface. Right: discharge.}
  \label{fig:vac}
\end{figure}

\subsection{Thacker's Periodic Solutions}
\label{sec:thacker}

We present two exact solutions of (\ref{eq:cSWE}) proposed by Thacker
in \cite{Thacker1981}.  They both describe oscillations of a free
surface in a parabolic basin with a free shoreline. The basin is
defined as 
\begin{equation}
  \label{eq:bthacker}
  b(\vec x) = b(r_c) = -H_0 \left(1 - \frac{r_c^2}{a^2}\right).
\end{equation}
$r_c$ defines the distance from the basin's centre, $H_0$ the height
of the centre and $a$ is a parameter. We will define two functions
$f(\vec x, t)$ that describe solutions of (\ref{eq:cSWE}) with $h(\vec
x, t) = \max(f(\vec x, t) - b(\vec x), 0)$. Both test cases will be
computed on the domain $\Omega=[-2,2]^2$. The presented results have been
performed with a grid size of $\dx = 0.4$. Reference solutions can be
found in \cite{RicchiutoBollermann2009, RicchiutoBollermann2008,
  MarcheBonnetonFabrieEtAl2007}.

\subsubsection{Thacker's Curved Solution}
\label{sec:curved}

\begin{figure}
  \centering
  \hspace*{\fill}
  \includegraphics{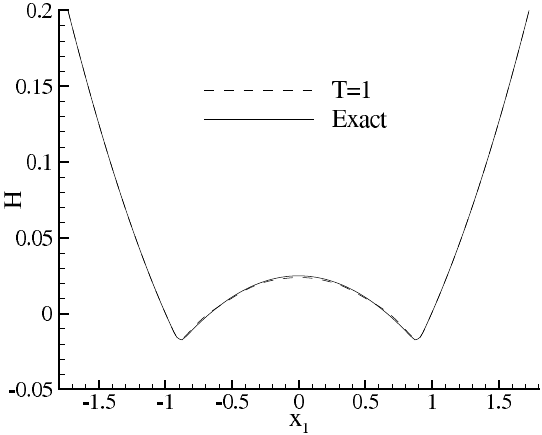}
  \hfill
  \includegraphics{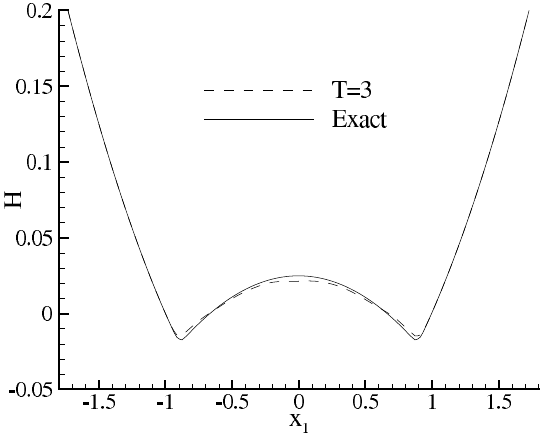}
  \hspace*{\fill} 
  \caption{Thackers curved solutions. Left: $t=T$. Right: $t=3T$}
  \label{fig:curved}
\end{figure}

The first function results in a curved oscillation over $b$, it reads
\begin{equation}
\label{eq:curvedf}
  f(r_c,t) = H_0\left(-1 + \frac{\sqrt{1-A^2}}{1-A\cos(\omega t)} 
      - \frac{r_c^2}{a^2}\left( 1- \frac{1-A^2}{(1-A\cos(\omega t))^2}\right)\right).
\end{equation}
Here, $\omega = \sqrt{8gH_0/a^2}$ is the frequency and for a given
$r_0 > 0$, $A$ is the shape parameter
\[
A = \frac{a^2 - r_0^2}{a^2+r_0^2}.
\]
For the computation we set $a=1$, $H_0=0.1$ and $r_0=0.8$, which
results in an oscillating period of $T\approx 2.22$. The initial
velocity is set to $\vv_0 = (0,0)^{tr}$.

\begin{table}  
  \centering
  \begin{tabular}{cc cc cc c}\hline 
    \# cells & $L^\infty$  & EOC & $L^1$  & EOC& $L^2$  & EOC \\ \hline\hline
    25 $\times$ 25   & 9.8038e-03 & &1.4526e-02 & & 9.2884e-03&  \\ 
    50 $\times$50 &3.6191e-03&1.44 &3.9038e-03&1.90 &2.7762e-03&1.74 \\
    100 $\times$100 &1.5252e-03&1.25 &1.3127e-03&1.57 &9.5937e-04&1.53 \\
    200 $\times$200 &1.1820e-03&0.37 &4.6649e-04&1.49 &3.8549e-04&1.32 \\
    400 $\times$400 &5.3221e-04&1.15 &1.7806e-04&1.39 &1.4907e-04&1.37 \\ \hline
  \end{tabular}
  \caption{Experimental order of convergence (EOC) for Thackers curved solution. Error in
    water height in different norms.}
  \label{tab:curved}
\end{table}

\subsubsection{Thacker's Planar Solution}
\label{sec:planar}

The second solution is a planar surface rotating around the basin. The
corresponding function is
\begin{equation}
  \label{eq:planarf}
f(\vec x, t) = \frac{\eta H_0}{a^2}(-\eta + 2(\vec x - \vec x_C)\cdot
(\cos(\omega t), \sin(\omega t))^{tr})
\end{equation}
with $\omega = \sqrt{2gH_0/a^2}$ the frequency and $\eta$ another
parameter.  Here, we set $a=1$, $H_0=0.1$ and $\eta=0.5$. The
resulting period is then $T\approx 4.44$. The initial velocity in the
wetted domain is given as $\vv_0 = (0, \eta\omega)^{tr}$.

\begin{table}  
  \centering
  \begin{tabular}{cc cc cc c}\hline 
    \# cells & $L^\infty$  & EOC & $L^1$  & EOC& $L^2$  & EOC \\ \hline\hline
    25 $\times$ 25   & 3.3855e-02 & &4.6507e-02 & & 2.7148e-02&  \\ 
    50 $\times$50 &1.7455e-02&0.96 &1.8179e-02&1.36 &1.1660e-02&1.22 \\
    100 $\times$100 &1.0543e-02&0.73 &1.0486e-02&0.79 &6.6938e-03&0.80 \\
    200 $\times$200 &8.2376e-03&0.36 &8.3640e-03&0.33 &5.4658e-03&0.29 \\
    400 $\times$400 &7.1238e-03&0.21 &7.8559e-03&0.09 &5.1747e-03&0.08 \\ \hline
  \end{tabular}
  \caption{Experimental order of convergence (EOC) for Thackers planar solution. Error in
    water height in different norms.}
  \label{tab:planar}
\end{table}  

We present the water height after one and three oscillations along the
line $x_2=0$ in Fig.~\ref{fig:curved} for the curved solution and in
Fig.~\ref{fig:planar} for the planar solution. The exact solution is
very well reproduced, independent of the shape of the initial
solution. We can see a slight smearing after three periods for the
curved solution, where the maximum value at the centre is reduced and
the drying/wetting interface has been pushed outward. Similarly, the
interface of the planar solution has also moved a little bit inwards
at $t=3T$. Again, for both solutions there is no production of
spurious waves at the dry boundary.

In Tables~\ref{tab:curved} and \ref{tab:planar} we present a
convergence study for the two test cases. The experimental order of
convergence for the curved solution is well better than one, which
meets our expectations. The errors are slightly better than in
\cite{RicchiutoBollermann2008}.
\begin{figure}
  \centering
  \hspace*{\fill}
  \includegraphics{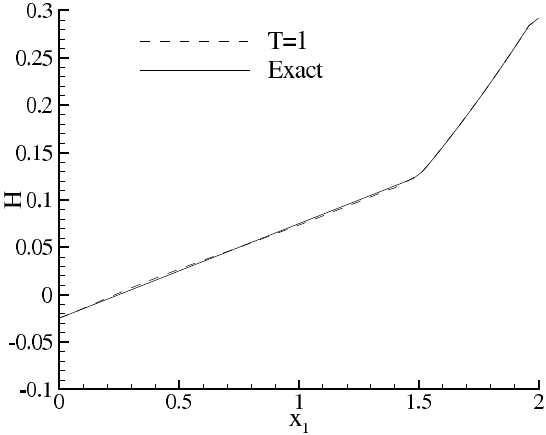}
  \hfill
  \includegraphics{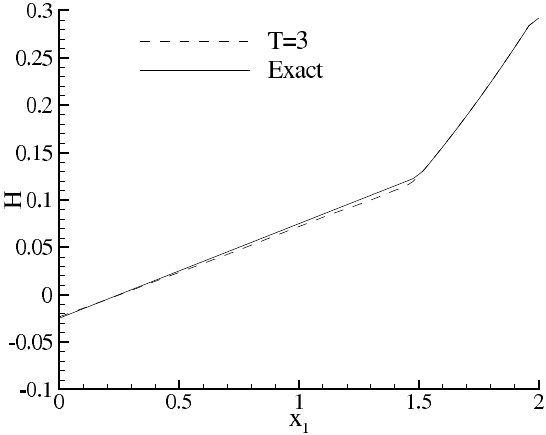}
  \hspace*{\fill}
  \caption{Thackers planar solutions. Left: $t=T$. Right: $t=3T$}
  \label{fig:planar}
\end{figure}
For the planar solution, however, the order quickly drops to zero.
The problem seems to raise from the boundary of the wetted domain,
where we have supersonic velocities tangential to the bottom slope. So
the problem might be related to the evolution operators, as they
produce inexact solutions in other supersonic situations as well, see
the discussion in Section~\ref{sec:entropy} and \ref{sec:conc}.

\subsection{Wave Run-up on a Conical Island}
\label{sec:island}

\begin{figure}
  \centering
  \hspace*{\fill}
  \includegraphics{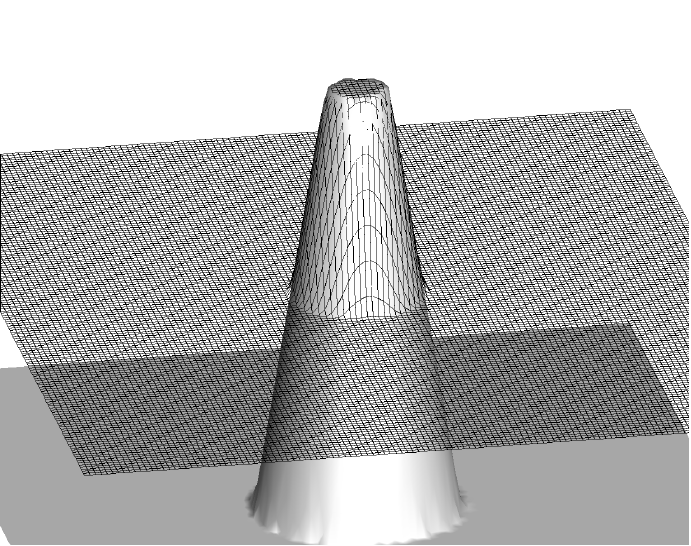}
  \hspace*{\fill}
  \includegraphics{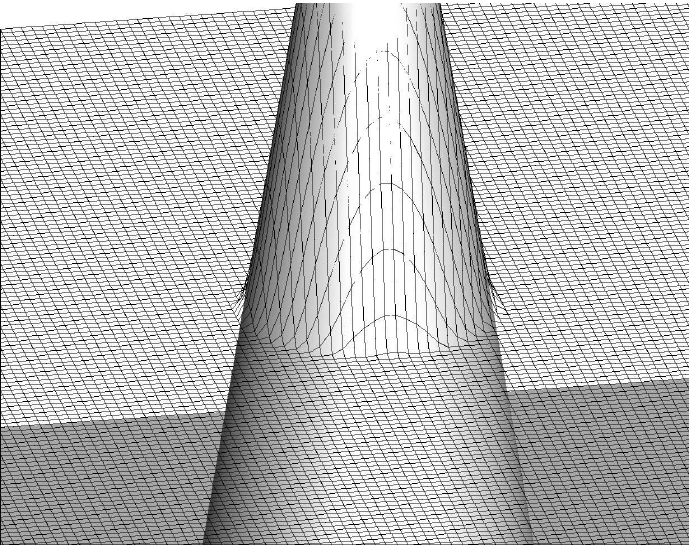}
  \hspace*{\fill}
  \caption{Circular island, lake at rest situation at $t=5$. Left:
    whole domain.  Right: zoom on island.}
  \label{fig:islandlar}
\end{figure}

In this case we simulate the run-up of a solitary wave over a conical
island. It has been performed experimentally at the U.S. Army Engineer
Waterways Experiment Station, see
\cite{BriggsSynolakisHarkinsEtAl1995, tsunami}. The computational
domain is $\Omega = [0, 25]\times[0, 30]$ and we set $\dx = 0.2$. The
centre of the island is located at $\vx_C = (12.5, 15)$ and with $r=
\|\vx - \vx_C\|$ its shape is given by
\begin{equation}
  \label{eq:island}
  b(r) =
  \begin{cases}
    0.625 & \text{$r \leq 1.1$} \\
    (3.6 - r)/4 & \text{$r \leq 3.6$} \\
    0 & \text{else.}
  \end{cases}
\end{equation}
The initial free surface is given by $H_0 = 0.32$. We start by giving
an example of the well-balancing capabilities of the scheme and
compute the lake at rest situation until $t=5$. The results are shown
in Fig.~\ref{fig:islandlar}. The lake at rest is perfectly preserved,
which is confirmed by the errors given in
Table~\ref{tab:islandlar}. For all the 3D views of this example, the
vertical axis representing the free surface was scaled by a factor of
25 to emphasise the results.

\begin{table}
  \centering
  \begin{tabular}{cccc}
    \hline
    & $L^\infty$ & $L^1$ & $L^2$ \\ \hline
    $e_H$ &4.44089e-16 &6.50361e-14 &3.14147e-15 \\
    $e_{v_1}$ &2.69215e-15 &2.53222e-13 &1.24917e-14 \\
    $e_{v_2}$ &2.92617e-15 &2.73056e-13 & 1.33148e-14\\
  \end{tabular}
  \caption{Lake at rest around a conical island. Errors at $t=5$.}
  \label{tab:islandlar}
\end{table}

\begin{figure}
  \centering
  \hspace*{\fill}
  \includegraphics{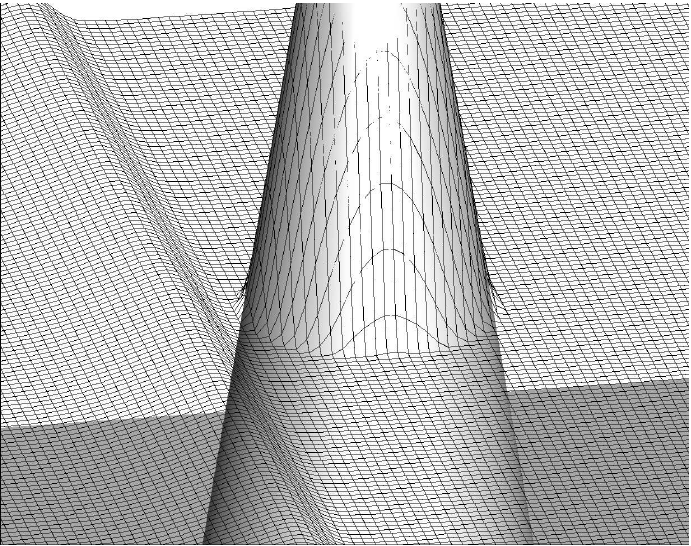}
  \hspace*{\fill}
  \includegraphics{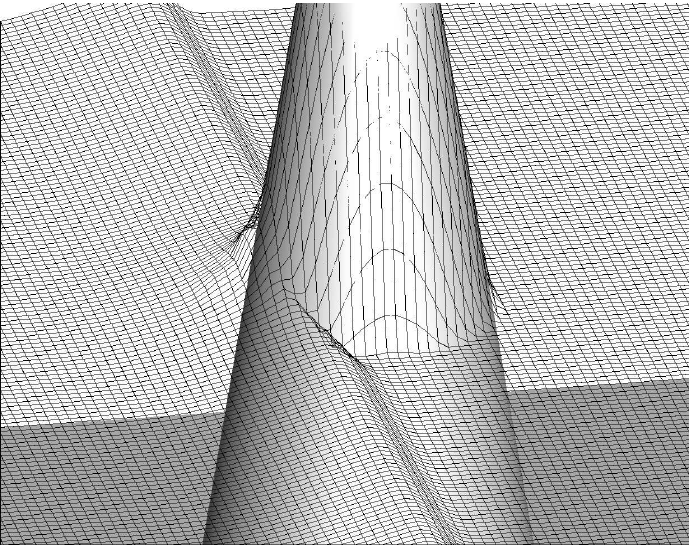}
  \hspace*{\fill}
  \caption{Run-up on a circular island. Left: wave approaching island, $t=7.9$.
    Right: run-up at front of the island, $t=9.1$.}
  \label{fig:island1}
\end{figure}

We now compute the actual wave where at time $t=0$ a wave enters the
computational domain at $x_1=0$. The height of the wave is given by
\begin{equation*}
  H(0,y,t) = H_0 + \alpha H_0 \left(\frac{1}{\cosh( \xi \sqrt{gH_0/L} (t-3.5)}\right)^2
\end{equation*}
with $L= 15$, $\alpha = 0.1$ and $\xi =
\sqrt{3\alpha(1+\alpha)L^2/(4H_0^2)}$, cf. \cite{HubbardDodd2002,
  MarcheBonnetonFabrieEtAl2007}.
\begin{figure}
  \hspace*{\fill}
  \includegraphics{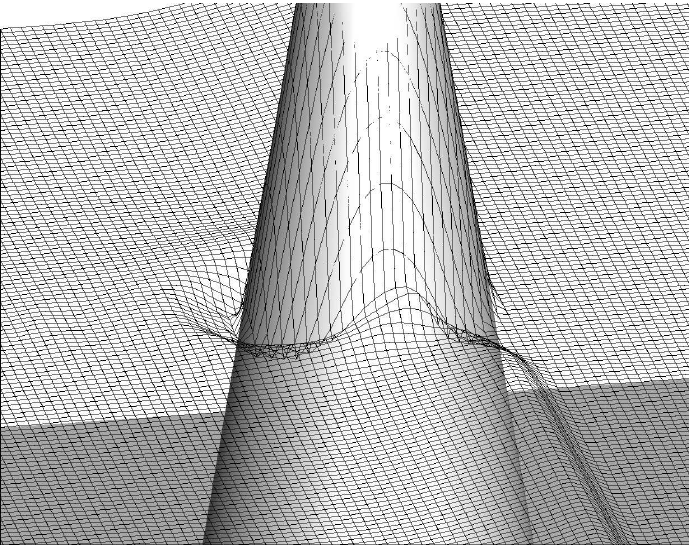}
  \hspace*{\fill}
  \includegraphics{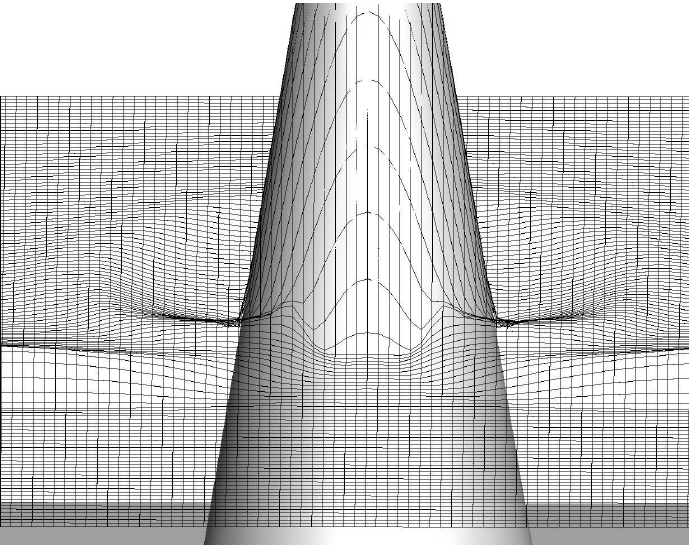}
  \hspace*{\fill}
  \centering
  \caption{Run-up on a circular island. Left: lateral run-up, $t=10.7$. Right:
    symmetric waves around the island, $t=12.1$.}
  \label{fig:island2}
\end{figure}
We present 3D views of the solution in Fig.~\ref{fig:island1},
\ref{fig:island2} and~\ref{fig:island3}.  Fig.~\ref{fig:island1} shows the wave
approaching the island and the instant of its maximal run-up. Here, as
in all the following figures, the region in front of the wave remains
completely unaffected, once again confirming the well-balancing of the
scheme. Due to the run-up, the wave is slowed down at the island,
finally resulting in a reflection of the wave presented in
Fig.~\ref{fig:island2}. This leads to the formation of two symmetric
waves surrounding the island with a noticeable run-up also on the
sides of the island. Away from the island, we observe the circular
reflected wave approaching the boundaries of the computational
domain. Behind the island, the lateral waves reunite and produce a
second peak in the run-up on the lee side of the island, see
Fig.~\ref{fig:island3}. Finally, the wave leaves the domain and the
surface returns to the lake at rest, as shown in the last picture.

\begin{figure}
  \hspace*{\fill}
  \includegraphics{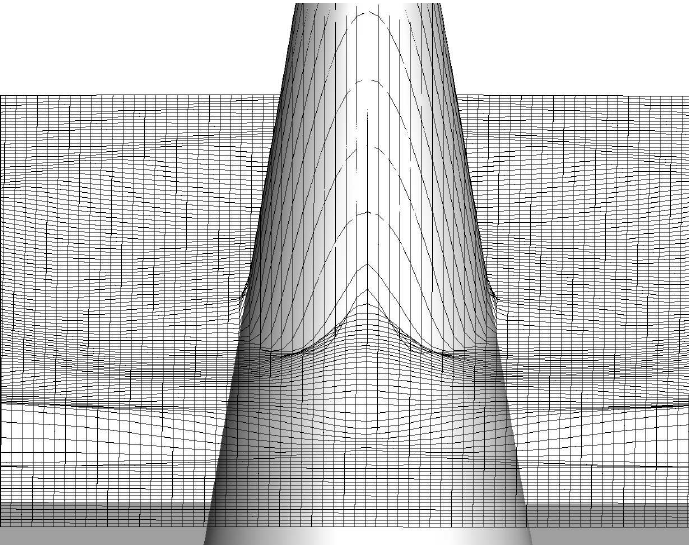}
  \hspace*{\fill}
  \includegraphics{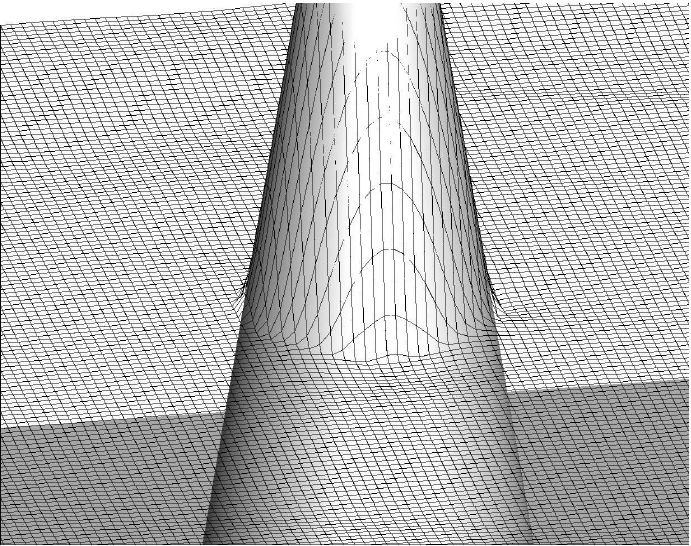}
  \hspace*{\fill}
  \centering
  \caption{Run-up on a circular island. Left: run-up behind the
    island, $t=13.7$. Right: reestablished lake at rest at $t=40$.}
  \label{fig:island3}
\end{figure}

\begin{figure}
\centering
  \includegraphics{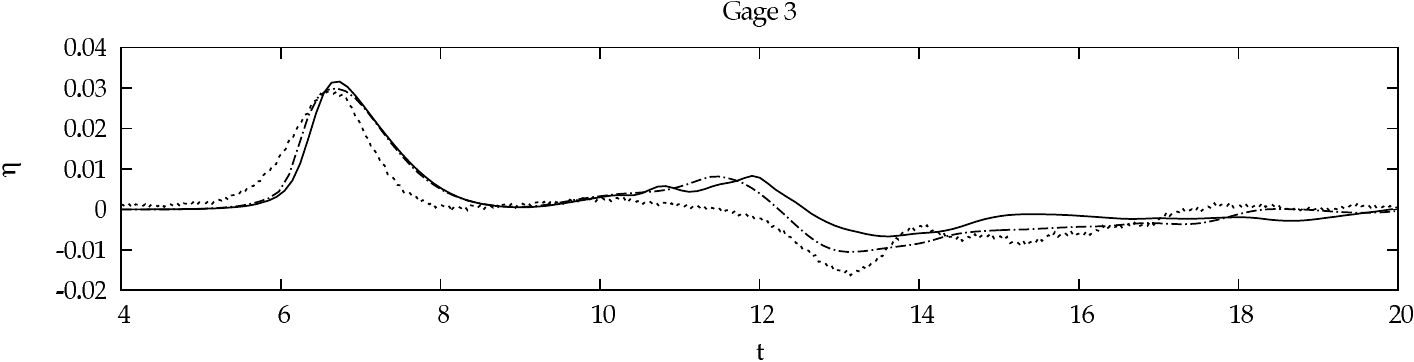}\\[2mm]
  \includegraphics{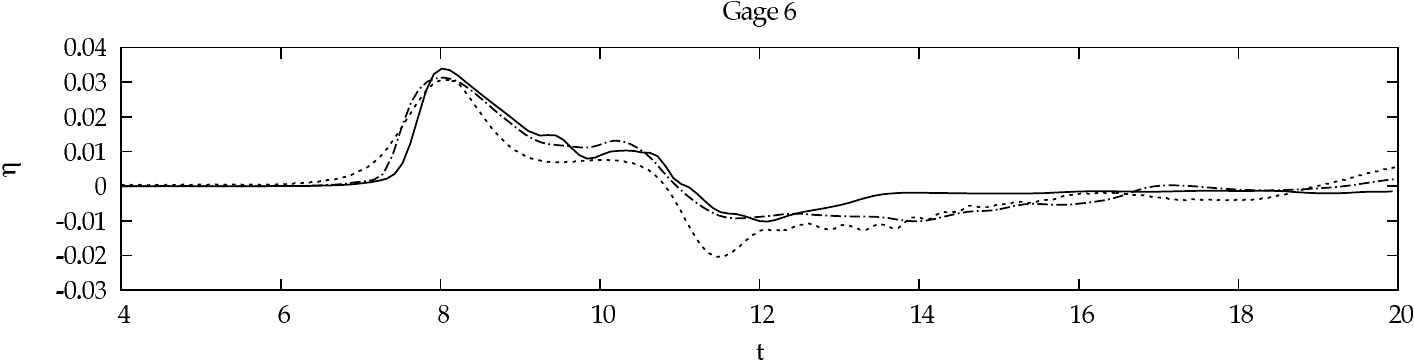}\\[2mm]
  \includegraphics{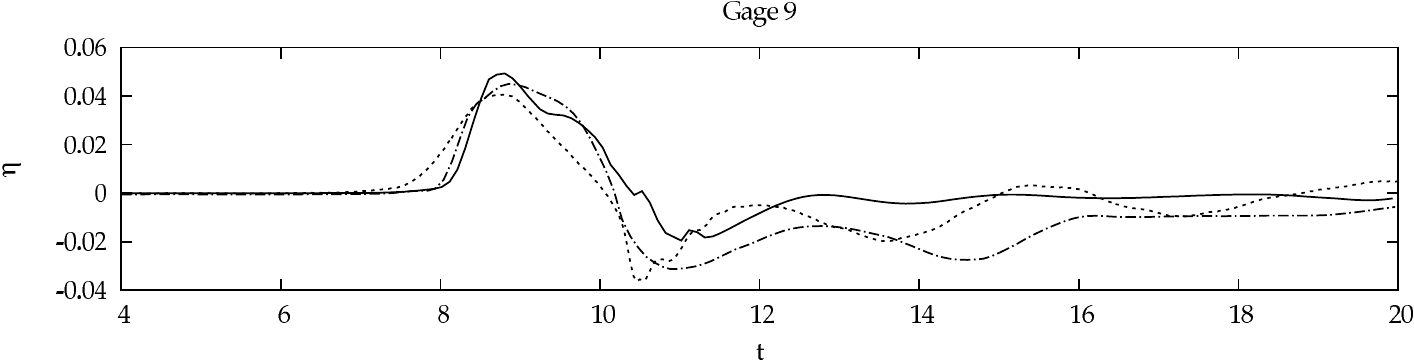}\\[2mm]
  \includegraphics{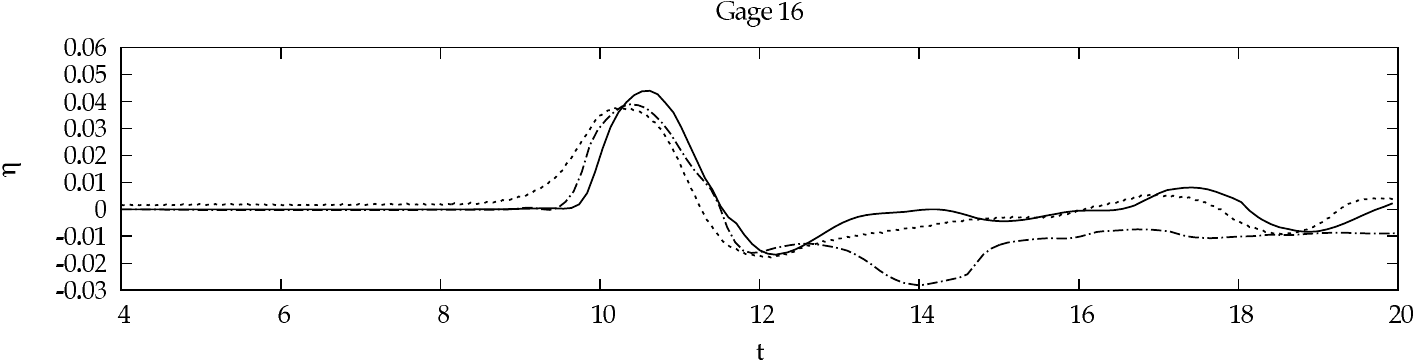}\\[2mm]
  \includegraphics{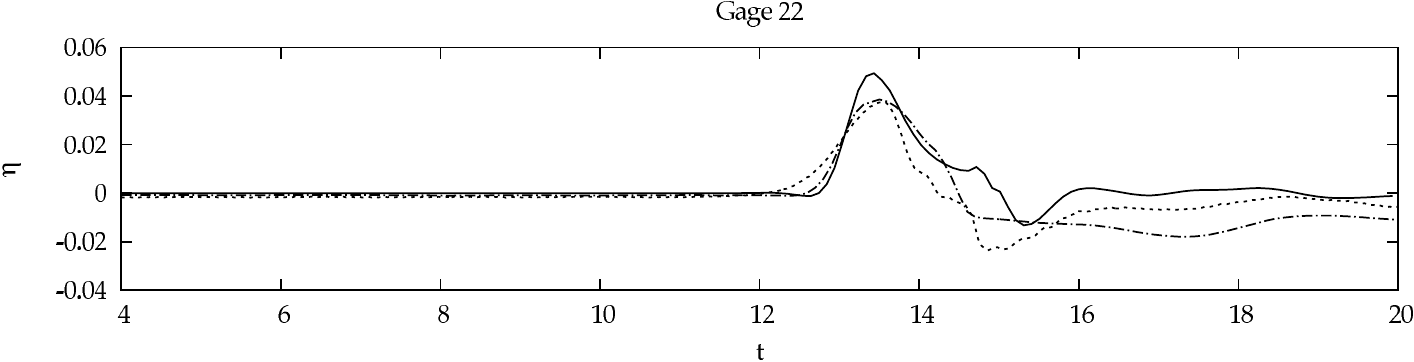}
  \caption{Run-up on a circular island. Time variation of the free
    surface $\eta = H - H_0$ at wave gages 6, 9, 16, and 22 of
    experiment from \cite{BriggsSynolakisHarkinsEtAl1995}. Dotted line:
    experimental data from \cite{tsunami}. Solid line: FVEG
    scheme. Dashed-dotted line: RD scheme from \cite{RicchiutoBollermann2009}}
  \label{fig:island_gages}
\end{figure}

\begin{figure}
\centering
  \includegraphics{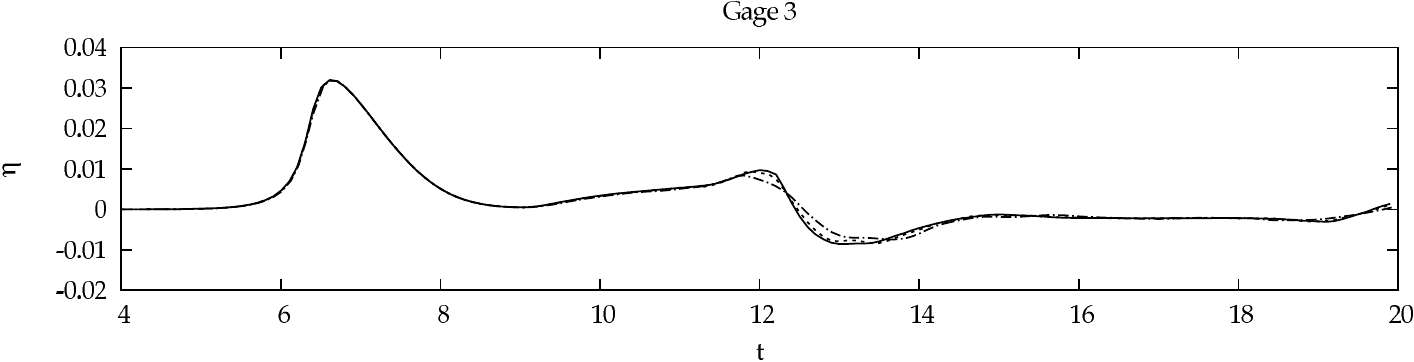}\\[2mm]
  \includegraphics{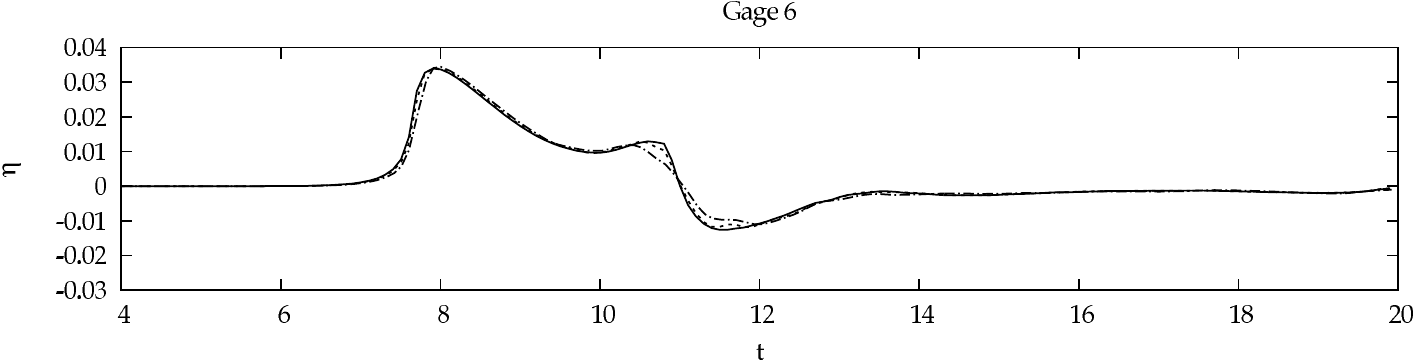}\\[2mm]
  \includegraphics{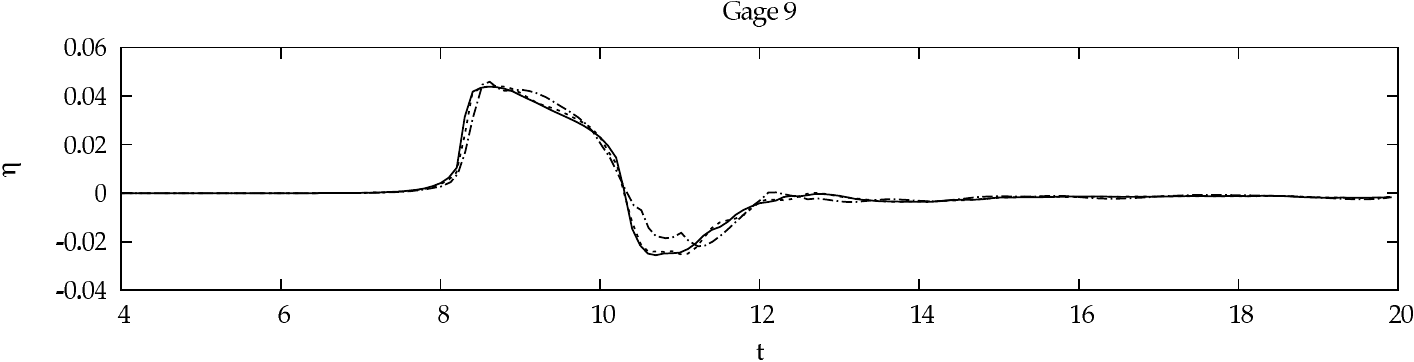}\\[2mm]
  \includegraphics{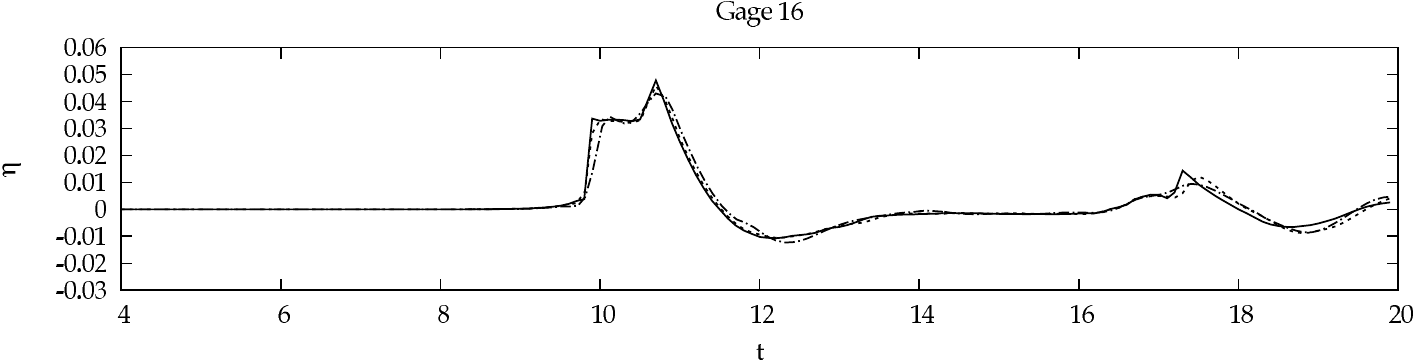}\\[2mm]
  \includegraphics{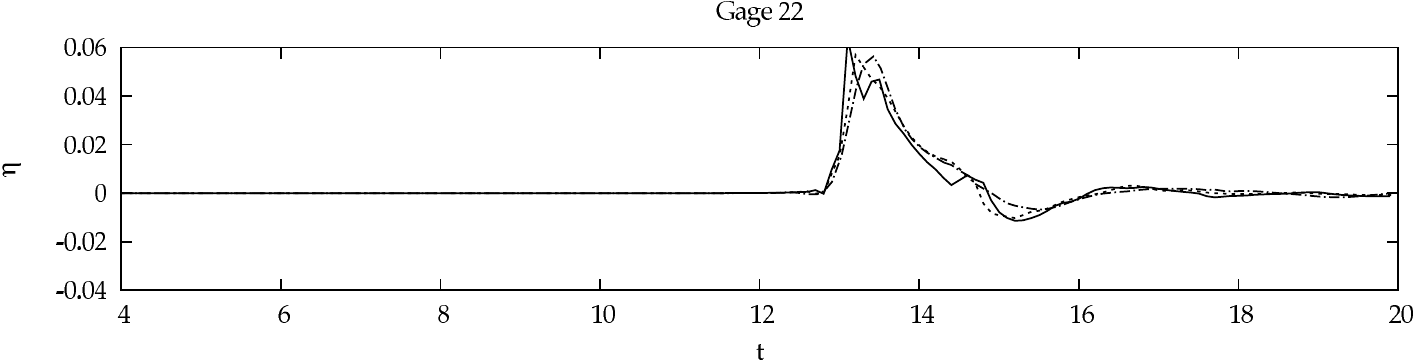}
  \caption{Run-up on a circular island. Time variation of the free
    surface $\eta = H - H_0$ at wave gages 6, 9, 16, and 22 of
    experiment from \cite{BriggsSynolakisHarkinsEtAl1995}, convergence of
    FVEG scheme. Dashed-dotted line: $\dx = 1/8$ Dotted line: $\dx = 1/16$
     Solid line: $\dx = 1/32$}
  \label{fig:island_gages_conv}
\end{figure}


Moreover we show the evolution of the free surface at chosen points in
Fig.~\ref{fig:island_gages}. The position of the gages is given by
$\vx_3 = (6.36, 14.25)$, $\vx_6 = (8.9, 15)$, $\vx_9 = (9.9, 15)$,
$\vx_{16} = (12.5, 12.42)$ and $\vx_{22} = (15.1, 15)$. For comparison
we also display the measured data obtained from \cite{tsunami} as well
as the results from the residual distribution (RD) scheme presented in
\cite{RicchiutoBollermann2009}, which have been computed on a
comparable grid\footnote{The unstructured triangulation used in
  \cite{RicchiutoBollermann2009} consists of 19824 elements, whereas
  the grid used here has 18750 cells. }. Both numerical schemes
produce steeper wave fronts than the experiment, which results from
the lack of dissipation in the shallow water model. The FVEG scheme
shows a stronger steepening than the RD scheme, and also predicts
slightly higher waves at all gages. At gages 3 and 6, the run-down
process is pretty similar with both schemes, while the FVEG scheme
returns quicker to a constant water height. Both schemes produce less
accurate results at gage 9. The FVEG line is shifted to the upper
left, thus resulting in a less pronounced run-down and again it
returns quickly to a constant height. The RD scheme provides smoother
results with a better approximation of the maximal run-down, but stays
below the original water height for a longer period. At gage 16, the
FVEG scheme gives a quite accurate representation of the wave, where
the RD scheme introduces an undershoot after the first wave. At the
last gage, the graph of the FVEG scheme is similar to the measurement,
but somewhat shifted to the upper right. This is probably due to the
over-prediction of the maximal wave height. The RD schemes gives a
better approximation of the maximal height, but the following graph is
smoothed out stronger. Finally, in Fig.~\ref{fig:island_gages_conv},
we present solutions at the gages for different grid resolutions,
computed with the FVEG scheme. We can clearly see that the scheme
converges. The steepening of the fronts becomes more pronounced and
the perturbations during the run-down vanish on the finer
grids. However, the main features are already captured on the
relatively coarse grid used for the simulation in
Fig.~\ref{fig:island_gages}. All in all, the FVEG scheme produces a
good reproduction of the wave, and the results are clearly competitive
to other numerical schemes like the ones presented in
\cite{HubbardDodd2002,RicchiutoBollermann2009}.

\section{Conclusions and Outlook}
\label{sec:conc}

We presented an approach to ensure positivity of the water height for general
finite volume schemes without affecting the global time step. This was achieved
by limiting the outgoing fluxes of a cell whenever they would create negative
water height. Physically, this corresponds to the absence of fluxes in the presence
of vacuum. A splitting of advective and gravity driven parts of the flux
preserved the well-balancing. In the context of FVEG schemes, we applied these
techniques to develop a positivity preserving scheme which is well-balanced
in the presence of dry areas. The scheme can also properly
handle sonic rarefaction waves, thanks to a new entropy correction based on the
evolution operators. We tested the scheme on a number of problems and in general
obtained satisfying results.

However, the discussion of the entropy fix (see Section~\ref{sec:entropy})
revealed that  in supersonic or transonic regimes the linearised wave cones used
in the EG operator do not reflect the physical domain of dependence adequately.
We conjecture that this is the origin of the loss of convergence for Thacker's 
planar solution (see Section~\ref{sec:planar}), since here the velocities
tangential to the boundary are larger than the (vanishing) gravitational speeds.
Two issues should be analysed further. As mentioned above, the first is the
linearisation strategy used in (\ref{eq:lSWE}). With the entropy fix from
\ref{sec:entropy} we made a first step towards a more sophisticated strategy
adapted to the state of flow. The other issue is related to the approximation of
the resulting linearised evolution operators. The approximations used here and
in \cite{LukacovaNoelleKraft2007} are based on the approximations from
\cite{LukacovaMortonWarnecke2004}, where they have been developed for the wave
equations. Now for this system the second eigenvalue is always zero, so the
sonic cone is never shifted in space with respect to the prediction point. An
approximation taking this shift into account should give more accurate results
in the critical regime.

Another possibility to improve the results is the introduction of friction
terms. This could be helpful to control the velocities at the dry
boundary by slowing down the waves near the shoreline. Finally we will
combine the new scheme with the adaptation techniques presented in
\cite{BollermannLukacovaNoelle2009}.

\section*{Acknowledgements}

This work was supported by DFG-Grant NO361/3-1 ``Adaptive semi-implicit FVEG
methods for multidimensional systems of hyperbolic balance laws''.

\bibliographystyle{abbrv}
\bibliography{/home/externe/andreas/bib/arbeit}

\end{document}